 \documentclass[11pt]{ip-journal-jdg}
\usepackage[dvips,final]{graphicx}
\usepackage[dvips]{geometry}
\usepackage{pst-node}
\usepackage{graphicx}

\usepackage{epsfig}

 \usepackage{latexsym}
 \usepackage{amssymb}
 \usepackage{graphicx}

 \renewcommand{\Re}{{\operatorname{Re}\,}}
 \renewcommand{\Im}{{\operatorname{Im}\,}}

 \renewcommand{\epsilon}{\varepsilon}
 \newcommand{\dist}{{\operatorname{dist}}}

 \newcommand{\T}{{\mathbf T}^n}

 \newcommand{\R}{{\mathbb R}}
 \newcommand{\C}{{\mathbb C}}
 
 \newcommand{\Z}{{\mathbb Z}}

 \newcommand{\supp}{{\operatorname{Supp\,}}}
 \renewcommand{\phi}{\varphi}

 \newcommand{\gcal}{\mathcal{G}}

 \newtheorem{theo}{{\sc Theorem}}[section]
 
 \newtheorem{cor}[theo]{{\sc Corollary}}

 \newtheorem{lem}[theo]{{\sc Lemma}}
 
 \newtheorem{prop}[theo]{{\sc Proposition}}
 
 \newenvironment{rem}{\medskip\noindent{\it Remark:\/} }{\medskip}

 \newtheorem{defn}[theo]{{\sc Definition}}
 
 \newcommand{\sgn}{\operatorname{sgn}}

 \newcommand{\cs}{$\clubsuit$}

 \newcommand{\eo}{\epsilon_{\circ}}

 \newcommand{\h}{h}

 \newcommand{\qc}{q^{\mathbb{C}}}

\title[nodal intersection]{Intersection bounds for nodal sets of planar Neumann eigenfunctions with interior analytic curves}

\begin{document}
\setcounter{page}{1}
\author{Layan El-Hajj}\address{Department of Mathematics and Statistics, McGill University, Montreal, Canada }\email{elhajj@math.mcgill.ca}
\author{John A. Toth}\address{Department of Mathematics and Statistics, McGill University, Montreal, Canada }
 \email{jtoth@math.mcgill.ca} \thanks{The second author's research was partially supported by NSERC grant \# OGP0170280. This manuscript was partly written during May, 2012 when the second author was visiting Beijing. He wishes to acknowledge the financial support and hospitality of the  Mathematical Sciences Center of Tsinghua University.
}

\maketitle

\begin{abstract} Let $\Omega \subset \R^2$ be a bounded piecewise smooth domain and  $\phi_\lambda$ be a Neumann (or Dirichlet) eigenfunction with eigenvalue $\lambda^2$ and nodal set ${\mathcal N}_{\phi_{\lambda}} = \{ x \in \Omega; \phi_{\lambda}(x) = 0 \}.$  Let $H \subset \Omega$ be an interior $C^{\omega}$ curve.  Consider the intersection number
$$ n(\lambda,H):= \#  ( H \cap {\mathcal N}_{\phi_{\lambda}} ).$$
We first prove that for general piecewise-analytic domains, and  under an appropriate ``goodness" condition on $H$   (see  Theorem \ref{mainthm1}),
\begin{equation}\label{estimate}
n(\lambda,H) = {\mathcal O}_H(\lambda) \,\, 
\end{equation}
as $\lambda \rightarrow \infty.$
Then, using Theorem \ref{mainthm1}, we prove in Theorem \ref{mainthm2} that the bound in  (\ref{estimate}) is satisfied in the  case of  quantum ergodic (QE) sequences of interior eigenfunctions, provided $\Omega$ is convex and $H$ has strictly positive geodesic curvature.
\end{abstract}

\section{Introduction}
 Let $\Omega \subset \R^2$ be a real analytic, bounded  planar domain with boundary $\partial \Omega$ and $H \subset \mathring{\Omega}$ a real-analytic interior curve. We consider here the Neumann (or Dirichlet)
 eigenfunctions $\varphi_{\lambda}$ on real analytic plane domains $\Omega \subset \mathbb{R}^{2}$ with
\[
\left\{ {\begin{array}{*{20}c}
   { - \Delta \varphi _\lambda   = \lambda ^2 \varphi _\lambda  } & {\text{in }\Omega }  \\
   {\partial _\nu  \varphi _\lambda   = 0} \,\,(Neumann), \,\,\,\, {\varphi_{\lambda} = 0} \,\,(Dirichlet) & {\text{on }\partial \Omega } . \\

 \end{array} } \right.
\]
\newline
The nodal set of  $\phi_{\lambda}$ is by definition
\[
N_{\varphi _\lambda  }  = \{ x \in \Omega :\varphi _\lambda  (x)= 0\} .\]
Our main interest here involves estimating from above the number of intersection points of the nodal
lines of Neumann eigenfunctions (the connected components of the nodal set)
with a fixed analytic curve $H$ contained in the interior of
the domain $\Omega.$ 
We define the intersection number for Dirichlet data along H  by
\begin{equation} \label{dirichlet}
n(\lambda,H) = \# \{ N_{\phi_{\lambda}} \cap H \}. \end{equation}
 We recall from \cite{TZ} that  an interior curve $H$ is said to be {\em good} provided for some $\lambda_0 >0$ there is a constant $C=C(\lambda_0)>0$ such that for all $\lambda \geq \lambda_0,$
  \begin{equation} \label{Good}
  \int_H |\phi_{\lambda}|^2 d\sigma \geq e^{-C \lambda}. \end{equation}
   Assuming the goodness condition $(\ref{Good}),$ it is proved in \cite{TZ}   that
  \begin{equation} \label{tzbound} n(\lambda,H) = {\mathcal O}_H(\lambda).\end{equation}


It follows from unique continuation for the interior eigenfunctions and the potential layer formula  $\phi_{\lambda}(x) = \int_{\partial \Omega} N(x,r(s);\lambda) \phi_{\lambda}(s) d\sigma(s); \,\, x \in \text{int}(\Omega),$ that (\ref{Good}) is satisfied in the special case where $H=\partial \Omega.$ 
The goodness property (\ref{Good}) seems very likely generic (see \cite{BR}).  However, it is difficult to prove in concrete
examples that the same upper bound  is
satisfied for  all  eigenfunctions with $\lambda \geq \lambda_0.$  Indeed, in \cite{TZ},
only the special curve $\partial \Omega$ is shown to be good.
Recently, Jung \cite{Ju} has shown that in the boundaryless case,
closed horocycles  of hyperbolic surfaces of finite volume are good in the sense of
(\ref{Good}) and hence satisfy the ${\mathcal O}(\lambda)$ upper
bounds. In the case of flat 2-torus, Bourgain and Rudnick
\cite{BR} have recently proved $\lessapprox \lambda$ upper
bounds when $H$ is real-analytic with nowhere vanishing curvature
(they also prove $\gtrapprox \lambda^{1-\epsilon}$ lower bounds in
the case where $H$ is real-analytic and non-geodesic).

 \,\,   Despite these results, it is  clear that not all curves are good in the sense of (\ref{Good}). As a counterexample, consider the Neumann problem in the unit disc.  The eigenfunctions in polar variables $(r,\theta) \in (0,1] \times [0,2\pi]$ are $\phi_{m,n}^{even}(r,\theta) = C_{m,n} \cos m\theta J_{m}(j'_{m,n}r)$ and $\phi_{m,n}^{odd}(r,\theta) = C_{m,n} \sin m\theta J_{m}(j'_{m,n}r).$ Here, $J_m$ is the $m$-th integral Bessel function and $j'_{m,n}$ is the $m$-th critical point of $J_m$. The eigenvalues are $\lambda_{m,n}^2 = (j'_{m,n})^2$.  Fix $m \in \Z^+$ and consider $$ H_{m} = \{ (r, \theta); \theta = \frac{2\pi k}{m}; \, k=0,..., m-1 \}.$$
Then, clearly for any $n=0,1,2,...$ $\phi_{m,n}^{odd}|_{H_m} = 0$ and so
in particular $H_m$ is not good in the sense of (\ref{Good}).

\smallskip

The point of this paper is threefold:
\begin{itemize}

\item[(i)] To give an alternative proof of the nodal intersection bound of Toth and Zelditch  for interior curves $H$ under an a revised goodness condition on $H$ (see Theorem \ref{mainthm1}).

\item[(ii)] To establish exponential lower and upper bounds (see Theorem \ref{mainthm3}) for the Grauert tube maxima of analytic continuations of restrictions of  quantum ergodic (QE) eigenfunctions to  positively-curved $H$ in annular subdomains of the complexification of $H.$  

\item[(iii)] To use the lower bounds in (ii) combined with (i) to explicitly identify a large class of {\em interior} analytic curves in planar billiards that satisfy the $n(\lambda,H) = {\mathcal O}_H(\lambda)$ intersection bounds for interior QE sequences of eigenfunctions. That is the content of Theorem \ref{mainthm2}.
\end{itemize}


Moreover, for both Theorems \ref{mainthm1} and \ref{mainthm2},  the $n(\lambda,H)$ upper bound is proved  using the frequency function method  of F.H. Lin combined with some semiclassical microlocal analysis, rather than the Jensen argument in \cite{TZ}.  Indeed,  the revised  goodness condition (see Theorem \ref{mainthm1}) that is needed for all our results follows here readily from the main frequency function bound for the number of complex zeros in a complex thickening of $H.$

Our first main theorem is:
\begin{theo} \label{mainthm1}
Let $\Omega$ be a bounded, piecewise-analytic domain and  $H \subset
\mathring{\Omega}$ an interior, $C^{\omega}$ curve with restriction map $\gamma_H:C^0(\Omega) \rightarrow C^0(H).$ Let $H_{\eo}^{\C}$ denote the complex radius $\eo >0$ Grauert tube containing $H$ as its totally real submanifold and
 $(\gamma_H \phi_{\lambda})^{\C}$ be the holomorphic continuation of $\gamma_H\phi_{\lambda}$ to $H_{\eo}^{\C}.$  Suppose the curve $H$ satisfies the {\em revised goodness condition}
 \begin{equation} \label{revgood}
  \sup_{z \in H_{\eo}^{\C}} | (\gamma_H \phi_{\lambda})^{\C}(z)| \geq e^{-C \lambda}. \end{equation}
 for some $C>0.$ Then, there is a constant $C_{\Omega,H}>0$ such that for all $\lambda \geq \lambda_0,$
$$ n(\lambda,H) \leq C_{\Omega,H} \lambda.$$
 \end{theo}

Our results here are inherently semiclassical and so we introduce the parameter $h$  which takes values in the  sequence $\lambda_j^{-1}; j=1,2,3,....$  By a slight abuse of notation, we denote the Neumann (or Dirichlet) eigenfunctions $\phi_{\lambda}$ by $\phi_h,$  and  write $n(h,H) := n(\lambda,H).$  The restrictions to $H$ are denoted by $\phi_h^H:= \gamma_H \phi_h$ where $\gamma_H: C^0(\Omega) \to C^{0}(H)$ is the restriction operator $\gamma_H f = f|_{H}$. In the special case where $H= \partial \Omega$ we denote the Neumann (resp. Dirichlet) boundary traces by  $\phi_h^{\partial \Omega} := \gamma_{\partial \Omega} \phi_h$ (resp. $\phi_h^{\partial \Omega} :=  \gamma_{\partial \Omega} h \partial_{\nu} \phi_h$).



Our second result deals with the case of quantum ergodic sequences of eigenfunctions. We recall that given a piecewise smooth manifold $\Omega$ with boundary, a sequence of $L^2$-normalized eigenfunctions $(\phi_{h_{j_k}})_{k=1}^{\infty}$ is  {\em quantum ergodic} (QE) if for any $a \in S^{0}(T^*\Omega)$ with $\pi ( \text{supp}(a)) \subset \text{Int}(\Omega),$
 $$ \langle Op_{h_{j_k}}(a) \phi_{h_{j_k}}, \phi_{h_{j_k}} \rangle \sim_{h_{j_k} \to 0^+} \int_{S^*\Omega} a(x,\xi) d\mu,$$
 where $d\mu$ is Liouville measure.  By a theorem of Zelditch and Zworski \cite{ZZ}, for a domain with ergodic billiards, a density-one subset of eigenfunctions are quantum ergodic.  The domain $\Omega$ is {\em quantum uniquely ergodic} (QUE) if all subsequences are QE.\\
We also recall from \cite{TZ2} the {\em quantum ergodic restriction} (QER) problem, which is to determine conditions on a hypersurface $\Gamma$ so that the restrictions  $\gamma_{\Gamma}\phi_{h_{j}}$ to $\Gamma$ on a Riemannian manifold $(M, g)$  with ergodic geodesic  flow are quantum ergodic along $\Gamma$.

 An important consequence of Theorem \ref{mainthm1} concerns convex billiards.

  \begin{theo} \label{mainthm2} Let $\Omega$ be a bounded, piecewise-analytic convex domain and $H$ be a   $C^{\omega}$ interior curve with strictly positive geodesic curvature. Let $(\phi_{h_{j_k}})_{k=1}^{\infty}$ be a QE sequence of  Neumann or Dirichlet eigenfunctions in $\Omega.$ Then,
  $$n(h_{j_k},H) = {\mathcal O}_{H,\Omega}(h_{j_k}^{-1}).$$
  \end{theo}
  
   The proof of Theorem \ref{mainthm2} follows by showing that positively curved $H$ are good in the sense of (\ref{revgood}) and then applying Theorem \ref{mainthm1}.

 When $\Omega$ is a convex ergodic billiard, it follows from the QE result of Zelditch and Zworski \cite{ZZ} that Theorem \ref{mainthm2} applies to at least a density-one subsequence of eigenfunctions. To our knowledge, it is an open question as to whether or not there are ergodic billiards that are QUE.

  In the course of proving goodness for curved $H$, we actually prove a much stronger result (see Proposition \ref{mainthm2prop} and Theorem \ref{mainthm3} below).  We show that the Grauert tube maximum of complexified eigenfunction restrictions $ \max_{z \in H_{\eo}^{\C}} | (\gamma_H\phi_{h}^{\C}(z)|$ is in fact exponentially  increasing in $h$ provided the Grauert tube radius $\eo >0$ is sufficiently small. 
 We summarize this in the following theorem which seems of independent interest.
 
To state our next result, we consider  weight function 
$$S(t):= \max_{s \in [-\pi, \pi]} \Re (i \rho^{\C}(t,s))$$
where $\rho^{\C}$ is the complexified distance function between $H$ and $\partial \Omega$ (see (\ref{complexified phase}) and (\ref{complexified phase2})). In Lemma \ref{weightlemma} we compute the asymptotics of $S(t)$ for $\Im t \in [\eo-\delta, \eo],$ with $\eo >0$ small and $0 < \delta < \eo.$ We show that in such thin strips in the upper half-plane,
\begin{equation} \label{weight asymptotics}
S(t) = \Im t + \frac{\kappa^2_H(\Re t)}{6} (\Im t)^3 + O(|\Im t|^5). \end{equation}

\begin{theo} \label{mainthm3}
Let $\Omega$ be a convex bounded planar domain,  $H \subset \mathring{\Omega}$ an interior $C^{\omega}$ strictly convex closed curve with curvature $\kappa_H>0$ and $(\phi_{h_{j_k}})_{k=1}^{\infty}$  a QE sequence of interior eigenfunctions.  Then, for $\eo >0$ sufficiently small and any $ 0< \delta < \eo$ there there exist constants $C_1=C_1(\eo,\delta)>0$ and $C_2 = C_2(\eo)>0,$ such that for $ h \in (0,h_0(\eo)],$
$$ C_1 e^{m_H(\eo- \delta)/h} \leq \max_{z \in H^{\C}_{\eo}} | (\gamma_H \phi_h)^{\C}(z)| \leq  C_2 h^{- \frac{1}{2} } e^{M_H(\eo)/h}.$$
Here, $m_H(\eo - \delta):= \min_{t \in [-\pi,\pi] \times [\eo-\delta,\eo]} S(t)$ and $M_H(\eo) = \max_{t \in [-\pi,\pi] \times [0,\eo]} S(t).$ Moreover, it follows from (\ref{weight asymptotics})  that  $M_H(\eo) = \eo + O(\eo^3)$ and $m_H(\eo-\delta) = \eo-\delta + O( |\eo-\delta|^3).$ 

\end{theo}

\medskip
The lower bounds in Theorem \ref{mainthm3} are one of the main results  of this paper and use Proposition \ref{mainthm2prop} in a crucial way. To our knowledge, even the (much simpler) upper bounds for $|u_h^{H,\C}|$ are new for domains with boundary. General results for growth of   $\phi_h^{\C}$  for $C^{\omega}$ manifolds without boundary are proved in \cite{Z2}.

  Our nodal intersection bounds are consistent with S.T. Yau's famous conjecture on the Hausdorff measure of nodal sets \cite{BG,Do,DF,DF2,H,HL,HHL,HS,L,Y1,Y2} which asserts that for all smooth $(M,g)$ there are constants $c_1,C_1>0$ such that  $c_1 \lambda \leq | N_{\phi_\lambda}| \leq C_1 \lambda$, where $|\cdot|$ denotes Hausdorff measure. There has been important  recent progress on polynomial lower bounds in Yau's conjecture using several methods (see \cite{CM}, \cite{He}, \cite{Man}, \cite{SZ}). Contrary to the lower bounds on nodal length, there are no general nontrivial lower bounds for the intersection count studied here which is easily seen  by considering the disc (see also  \cite{JN,NJT,NS} and related results on sparsity of nodal domains \cite{Lew}). In analogy with the case of nodal domains, it is of interest to determine whether non-trivial (ie. polynomial in $\lambda$) lower bounds exist for nodal intersections under  appropriate dynamical assumptions (such as ergodicity) on the billiard  dynamics.  Recently, in \cite{GRS}, Ghosh, Reznikov and Sarnak have established such polynomial lower bounds in the case of arithmetic surfaces. We hope to return to this question elsewhere.

Throughtout the paper $C>0$ will denote a positive constant that can vary from line to line.

\subsection{Outline of the proof of Theorem 1.1}

We now describe the main ideas in the  proof of Theorem \ref{mainthm1} suppressing for the moment some of the technicalities. 
Let $q:[-\pi,\pi] \rightarrow H$ be a $C^{\omega}$-parametrization of the curve $H$  with $|q'(t)| \neq 0$ for all $t \in [-\pi,\pi]$ and let $r:[-\pi,\pi] \to \partial \Omega$ be the arclength parametrization of the boundary. We denote the respective eigenfunction restrictions (on the parameter domain) by $u_h^{H}(t) = \phi_h^H(q(t))$ and $u_h^{\partial \Omega}(s) = \phi_h^{\partial \Omega}(r(s)).$  As in \cite{TZ}, given the eigenfunction
restriction, $u_{h}^{H}(t) = \phi_h^H(q(t)), \, t \in [-\pi,\pi]$ the
first step is to complexify $u_{h}^{H}$ to a holomorphic function
$u_{h}^{H,\C}(t)$ with $t \in C_{\eo}$ where $C_{\eo}$ is a simply-connected
domain with $C^{\omega}$ boundary $\partial C_{\eo}$  containing the rectangle $S_{\eo,\pi}$ in the parameter space. The image of $S_{\eo,\pi}$ under the complexified parametrization of $H$ is the complex Grauert tube $H_{\eo}^{\C}$; that is, $H_{\eo}^{\C} := q^{\C}(S_{\eo,\pi}).$ The reason for
introducing the intermediate domain of holomorphy $C_{\eo,}$ is somewhat
technical and has to do with the frequency function approach to nodal
estimates, which is adapted to counting complex zeros in discs (see Lemma \ref{bdy lemma}).
Let $n(h, C_{\eo})$ denotes the number of complex zeros of $u_{h}^{H,\C}$ in the simply connected domain $C_{\eo}$.
The key frequency function estimate (see Proposition \ref{strip
bound})  gives the upper bound

\begin{equation} \label{outline1}
n(h,H) \leq n(h, C_{\eo}) \leq C_1 \left( \frac{ \|
\partial_{T} u_{h}^{H,\C} \|_{L^2_{\eo}} }{  \| u_{h}^{H,\C}
\|_{L^2_{\eo}}  } \right).
\end{equation}
Here, we write $L^2_{\eo}$ for $L^2(\partial C_{\eo},d\sigma(t))$ and $\partial_{T}$ is the unit tangential derivative along $\partial C_{\eo}.$ A key step in the proof of Theorem \ref{mainthm1} is to $h$-microlocally decompose the right hand side
in (\ref{outline1}). Let $\chi_R \in  C^{\infty}_{0}(T^*\partial C_{\eo} )$
with $\chi_R(\sigma) = 1$ for $|\sigma | \leq R+1$ and
$\chi_R(\sigma) = 0$ for $|\sigma | \geq R + 2$ with $R>0$
arbitrary.
Clearly,
\begin{equation} \label{outline2}
  \| \partial_{T} u_{h}^{H,\C} \|_{L^2_{\eo}} \leq \| \partial_{T} Op_h(\chi_R) u_{h}^{H,\C} \|_{L^2_{\eo}}   +           \| \partial_{T} (1- Op_h(\chi_R)) u_{h}^{H,\C} \|_{L^2_{\eo}}.
\end{equation}
For the first term on the right hand side of (\ref{outline2}), since $  h
\partial_T Op_h(\chi_R) \in Op_h(S^{0,0}(T^*\partial C_{\eo})),$ we have by
$L^2$-boundedness that
\begin{equation} \label{outline3}
\frac{ \| \partial_{T} Op_h(\chi_R) u_{h}^{H,\C} \|_{L^2_{\eo}}}{\| u_{h}^{H,\C} \|_{L^2_{\eo}}}  =  h^{-1} \frac{\| h \partial_{T} Op_h(\chi_R) u_{h}^{H,\C} \|_{L^2_{\eo}} }{\| u_{h}^{H,\C} \|_{L^2_{\eo}}} \leq C_{2} h^{-1} . \end{equation}
 As for the second term on the right hand side of (\ref{outline2}), by using potential layer formulas and the Cauchy-Schwarz inequality combined with a complex contour deformation argument (see Proposition \ref{decay estimate}), we show that
\begin{equation} \label{outline4}
 \|h\partial_{T} (1-Op_h(\chi_R)) u_{h}^{H,\C} \|_{L^2_{\eo}}  =  {\mathcal O}(e^{-C_R/h})\cdot \|u_h^{\partial\Omega}\|_{L^2}. \end{equation}
 Here, $C_R \gtrapprox R$ as $R \to \infty$ and  $L^{2}_{0} = L^2 ([-\pi,\pi], dt),$  so the term on the right hand side of (\ref{outline4}) involves the $L^2$-integral of the restriction of $\phi_h$ to the domain boundary $\partial \Omega.$\\
 Since $\|u_h^{\partial\Omega}\|_{L^2} = O(h^{-\alpha})$ for some $\alpha >0$, it follows that
 \begin{equation} \label{outline5} \|h\partial_{T} (1-Op_h(\chi_R)) u_{h}^{H,\C} \|_{L^2_{\eo}}  =  {\mathcal O}(e^{-C_R/h}).
 \end{equation}
From the Cauchy integral formula, Cauchy-Schwarz and the goodness condition $(\ref{revgood})$ we get
\begin{equation}
\label{outline6}\| u_{h}^{H,\C} \|_{L_{\eo}^2} \geq C \cdot \sup_{t \in S_{\eo,\pi}} | u_{h}^{H,\C}(t)| \gtrapprox e^{-C_0/h}.
\end{equation}
 From (\ref{outline6})  and (\ref{outline5}),
\begin{equation} \label{outline7}
 \frac{\| h\partial_{T} (1-Op_h(\chi_R) u_{h}^{H,\C} \|_{L^2_{\eo}}}{\| u_{h}^{H,\C} \|_{L_{\eo}^2}}  =  {\mathcal O}(e^{(-C_R + C_0) /h}) . \end{equation}
By choosing $R $ sufficiently large in the radial frequency cutoff $\chi_R$, we get that $C_R - C_0  \gtrapprox R  > 0$ and so, substitution of the estimates (\ref{outline2}), (\ref{outline3}) and (\ref{outline7}) in (\ref{outline1})
   completes the proof of Theorem \ref{mainthm1}. \qed

\subsection{Outline of the proof of Theorem 1.3}  \label{mainthm2proof}
 Let $H^{\C}_{\eo}$ be the complex Grauert tube of radius $\eo >0$ with totally-real part $H$ and let $\delta >0$ be an arbitrarily small constant. Choose $\zeta_{\eo} \in C^{\infty}_0(H^{\C}_{\eo}; [0,1])$ to be a cutoff in the Grauert tube equal to $1$ on  $H^{\C}_{\eo-2\delta} - H^{\C}_{\eo-3\delta}$   and vanishing outside $H^{\C}_{\eo-\delta} - H^{\C}_{\eo-4\delta}$. Let $\chi_{\eo,\delta}(t):= \zeta_{\eo,\delta}(q^{\C}(t))$ be the corresponding cutoff in the parameter domain.

  Ignoring technicalities arising from corner points on the boundary,  the main technical part of the proof of Theorem \ref{mainthm3} (see Proposition \ref{mainthm2prop} and Corollary \ref{propcor}) consists of  showing that under the non-vanishing curvature condition on $H$ and for $\eo>0$ small,  there is an order-zero semiclassical pseudodifferential operator $P(h) \in Op_h(C^{\infty}_0(B^*\partial \Omega))$ such that \begin{equation} \label{mainthm2.1}
h^{-1/2} \int \int_{\C/2\pi\Z} e^{-2 S(t)/h} | u_{h}^{H,\C}(t)|^2 \, \chi_{\eo,\delta}(t) \, dt d\overline{t} \sim_{h \rightarrow 0^+} \langle P(h) \phi_{h}^{\partial \Omega}, \phi_h^{\partial \Omega} \rangle. \end{equation}
Moreover, the principal symbol $\sigma(P(h))$ satisfies
\begin{equation} \label{mainthm2.2}
\int_{B^*\partial \Omega} \sigma(P(h)) \gamma^{-1} \, dy d\eta \geq C_{H,\Omega,\eo,\delta} >0\end{equation}
where $\gamma(y,\eta) = \sqrt{1-|\eta|^2}.$

Given  a quantum ergodic sequence $(\phi_{h_{j_k}})_{k=1}^{\infty}$, it follows that the boundary restrictions $(\phi_{h_{j_k}}^{\partial \Omega})_{k=1}^{\infty}$ are themselves quantum ergodic \cite{Bu,HZ}  in the sense that
\begin{equation} \label{qer}
 \langle P(h) \phi_h^{\partial \Omega}, \phi_h^{\partial \Omega} \rangle \sim_{h \rightarrow 0^+} \int_{B^*\partial \Omega} \sigma(P(h)) \gamma^{-1} \, dy d\eta. \end{equation}
 It then follows from (\ref{mainthm2.1}), (\ref{mainthm2.2}) and (\ref{qer}) that
\begin{equation} \label{mainthm2.3}
h^{-1/2} \int \int_{\C/2\pi\Z} e^{-2 S(t)/h} | u_{h}^{H,\C}(t)|^2 \, \chi_{\eo,\delta}(t) \, dt d\overline{t} \sim_{h \rightarrow 0^+} \int_{B^*\partial \Omega} \sigma(P(h)) \gamma^{-1} \, dy d\eta = C_{\Omega,H,\eo,\delta} >0.\end{equation}

To get the lower bounds for the Grauert tube maxima, we use the elementary  inequality
\begin{align} \label{easy}
h^{-1/2} \max_{t \in S_{\eo,\pi}} |u_h^{H,\C}(t)|^2 \times  \int \int_{\C/2\pi\Z} e^{-2S(t)/h} \chi_{\eo,\delta}(t) dt d\overline{t} & \nonumber \\
 \geq  h^{-1/2} \int \int_{\C/2\pi\Z} e^{-2 S(t)/h} | u_{h}^{H,\C}(t)|^2 \, \chi_{\eo,\delta}(t) \, dt d\overline{t}.& \end{align}
In view of (\ref{mainthm2.3}) and the formula $S(t) = \Im t + \frac{\kappa^2_H(\Re t)}{6} (\Im t)^3 + O(|\Im t|^5)$ in Lemma \ref{weightlemma}, the exponential lower bound for the Grauert tube maximum of $|u_h^{H,\C}|$ follows from (\ref{easy}). 

The upper bound follows by using the complexified potential layer formula 
$$u_{h}^{H,\C}(t) = \int_{\partial \Omega} N^{\C}(q^{\C}(t),r(s);h) u_{h}^{\partial \Omega}(r(s)) dr(s)$$ combined with Cauchy-Schwarz and the a priori bound $\| u_{h}^{\partial \Omega} \|_{L^2(\partial \Omega)} = O(1),$ which is a direct consequence of the QER property of the eigenfunction boundary traces. Further details are given in section \ref{eigenfnestimates}.
 \qed

\subsection{Outline of proof of Theorem \ref{mainthm2}} Theorem \ref{mainthm2} follows from Theorems \ref{mainthm3} and \ref{mainthm1} since the goodness condition (\ref{revgood}) for $H$ follows trivially from the lower bound in Theorem \ref{mainthm3}.  \qed

\medskip

\begin{rem}We note that the lower bound in Theorem \ref{mainthm3} gives exponential {\em growth}  for the Grauert tube maximum
$\max_{t \in S_{\eo,\pi}} |u_{h}^{H,\C}(t)|$ consistent with the upper bound. Therefore, it is much stronger than the goodness condition (\ref{revgood}) which only requires that Grauert tube maximum not {\em decay} faster than $e^{-C/h}$ for some $C>0.$ \end{rem}

\subsection{$h$-microlocal characterization of $P(h)$} \label{microlocal}

Although we give  a self-contained proof of Proposition \ref{mainthm2prop}  in section \ref{opanalysis}, it is useful to understand the $h$-microlocal rationale behind the characterization of $P(h)$ in the proposition. To simplify the argument somewhat, we continue to assume here that $\partial \Omega$ is smooth.

Let  $H_{\eo/2,\eo}^{\C} = H_{\eo}^{\C} - H_{\eo/2}^{\C}$ be the complex strip corresponding to $\eo/2 < \Im t < \eo$ in the complexification of $H$ and $\chi \in C^{\infty}_{0}(H_{\eo/2,\eo}^{\C}),$ a cutoff to the complex strip $H_{\eo/2,\eo}^{\C}.$ Given the composite operator $F_{\chi}(h): C^{\infty}(\partial \Omega) \to C^{\infty}_0(H_{\eo}^{\C}),$ with  $$F_{\chi}(h):= h^{-1/4}  e^{-S/h} \chi \gamma_H^{\C} N^{\C}(h),$$ the argument in Proposition \ref{mainthm2prop}  characterizes the operators
\begin{equation} \label{poperator}
P(h) = F_{\chi}(h)^* F_{\chi}(h): C^{\infty}(\partial \Omega) \to C^{\infty}(\partial \Omega) \end{equation}
as $h$-pseudodifferential of order $0.$

 The reason for this can be seen as follows. Under the positive curvature assumption on $H$ we  show in section \ref{normalform} that for $\eo >0$ small, there exists $s^* \in C^{\omega}(H_{\eo/2,\eo}^{\C})$ such that $s \mapsto \frac{\Re [ i \rho^{\C}(t,s)]}{\Im t}$ has a non-degenerate maximum at $s = s^*(t)$. Since the subharmonic weight function 
 $ S(t) = \Re i \rho^{\C}(t,s^*(t)),$ it follows that  the phase of the operator $F_{\chi}(h)$ is of  of the form
 \begin{equation} \label{realpart}
 \phi(t,s) = \Re \rho^{\C}(t,s) + i \beta_1(t,s) \Big( (s-s^*(t) )^2 + O(s-s^*(t))^3  \Big), \end{equation}
 where coefficient $\beta_1(t,s) \sim \kappa_H(Y(s))^2 \Im t$. Here, $s^*(\Re t,0) = Y^{-1}(\Re t)$ where $Y: \partial \Omega \to H$ is the glancing map relative to $H$ (see Definition \ref{grazing}). When H is strictly convex, up to choice of orientation, the map $Y$ is the $C^{\omega}$-diffeomorphism of $\partial \Omega$ with H that assigns to each point $r \in \partial \Omega$ the corresponding point $q \in H$ with the property that the geodesic ray joining these points is tangent to H at the terminal point, $q.$

 As for the real part of the phase $\Re \phi(t,s) = \Re \rho^{\C}(t,s),$ we note that in view of (\ref{realpart}),
  $WF_h '( F_{\chi}(h)) \subset \Lambda$ with 
\begin{align} \label{wavefront}
 \Lambda:&= \{ (s, d_s \rho^{\C}(t,s); t, -d_{t} \rho^{\C}(t,s) - i d_{t} S(t) ), \,\, s = s^*(t), \,\, \frac{\eo}{2} < \Im t < \eo \} \nonumber \\
 & \subset T^*\partial \Omega \times T^*H_{\eo}^{\C}. \end{align}
 One can decompose $\Lambda$ into real and imaginary parts $\Lambda = \Re \Lambda \oplus i \Im \Lambda $ with 
 \begin{align} \label{realwavefront}
 \Re \Lambda &= \{ (s, d_s \Re \rho^{\C}(t,s);  \Re t, -d_{\Re t}  \Re \rho^{\C}(t,s) ), \,\, s = s^*(t),  \,\, \frac{\eo}{2} < \Im t < \eo  \} \nonumber \\
& \,\,\,\, \subset T^*\partial \Omega \times T^*H. \end{align} 
Under the positive curvature condition on $H$, one can show that $(s, \Im t)$ are parametrizing coordinates for $\Re \Lambda.$  For $(x,\xi,y,\eta) \in \Re \Lambda$, the projections $\kappa_1(x,\xi,y,\eta) = (x,\xi)$ and $\kappa_2(x,\xi,y,\eta) = (y,\eta)$ are diffeomorphisms onto their images  and 
  the open real Lagrangian $\Re \Lambda$ is a canonical graph with respect to the symplectic form $\kappa_1^*(ds \wedge d\Im t) \oplus \kappa_2^*(- ds \wedge d\Im t)$ (see subsection \ref{graph}). In particular, $d_s d_{\Im t} \Re \rho^{\C}(t,s) \neq 0$ when $s = s^*(t).$
 As a consequence of this and the analysis of the imaginary part of the phase in (\ref{realpart}),  a key step in the proof of Proposition \ref{mainthm2prop} is the operator decomposition (see Lemma \ref{integralnf1})
 \begin{align} \label{fbi1}
 F_{\chi}(h)^* F_{\chi}(h) = U_{Y}(h)^* T(h)^* \, \chi^2  \,T(h) U_{Y}(h). \end{align}
 Here, $U_{Y}(h): C^{\infty}(\partial \Omega) \to C^{\infty}(H)$ is an $h$-Fourier integral operator quantizing the glancing diffeomorphism $Y: \partial \Omega \to H$. After an appropriate choice of complex variables $\tau_1 + i \tau_2 \in H_{\eo/2,\eo}^{\C}$ and identifying $H_{\eo/2,\eo}^{\C}$ with a subset of $T^*H$ via map $\tau_1 + i \tau_2 \to (\tau_1,\tau_2),$ the operators $T(h): C^{\infty}(H) \to C^{\infty}(T^*H)$ can be written in the form
 \begin{align} \label{fbi2}
 T(h)g (\tau_1,\tau_2) = (2\pi h)^{-3/4} \int_{H} e^{ [ i (\tau_1-u) \tau_{2} -  \tilde{\beta}(u,\tau_1,\tau_2)  | \tau_1 - u |^2 ]/h}  c(u,\tau_1,\tau_2;h) \, g(u) \, d\sigma(u)
 \end{align}
 where $c \sim \sum_{k=0}^{\infty} c_k h^k$ and $\tilde{\beta}(u,\tau_1,\tau_2) >0$ when $(u,\tau_1,\tau_2) \in \text{supp} \, c.$ Here,  supp $c(u,\cdot, \cdot;h) \subset \{ (\tau_1,\tau_2);  \frac{\eo}{2} < \tau_2 < \eo \}.$ Consequently, $T(h)$ is, $h$-microlocally on the image of $H_{\eo/2,\eo}^{\C},$  an FBI-transform in the sense of \cite{WZ} of order  zero. Indeed, the extra multiplicative factor of $h^{-1/4}$ is included in the definition of $F_{\chi}(h)$ to ensure that $T(h)$ is of order zero.   The identity (\ref{fbi1}) follows from the analysis in section \ref{opanalysis} (see, in particular, Lemma \ref{integralnf1}).
 
 Given (\ref{fbi1}), it follows by the  $h$-Egorov theorem for FBI-transforms (\cite{Zw} Theorem 13.12) that
 $$ T(h)^* \chi^2 T(h) = Q(h)$$
 where $Q(h): C^{\infty}(H) \to C^{\infty}(H)$ is an $h$-pseudodifferential operator on $H$ of order $0.$ Finally, conjugation with the $h$-Fourier integral operator $U_Y(h)$ quantizing the glancing map and another application of the usual $h$-Egorov theorem for $h$-pseudodifferential operators gives $P(h)$ in (\ref{poperator}) as an $h$-pseudodifferntial operator of order $0$ acting on the boundary $\partial \Omega.$ This is essentially the content of Proposition \ref{mainthm2prop} along with the explicit computation of the principal symbol of $P(h).$

\subsection{QER for Cauchy data along $H$ and growth of $u_h^{H,\C}$}

In (\ref{qer}) we have used that for Dirichlet, interior QUE for domains implies QE for the boundary traces $\phi_h^{\partial \Omega}.$ This follows from Burq's proof of boundary quantum ergodicity \cite{Bu} using the Rellich commutator argument (see also \cite{HZ} for a different proof) .  In the  Neumann case, the same is true as long as one uses test operators with symbols supported away from the tangential set to the boundary; in particular, our test operator $P(h)$ in (\ref{mainthm2.1}) has this property.  Neither statement is  necessarily  correct for the eigenfunction restrictions to a general interior  curve $H$ \cite{TZ2}.  An important point in this paper is that the nodal intersection count for an interior $H$ is linked to QER for the boundary values of eigenfunctions $\phi_h^{\partial \Omega},$ not  the QER problem for $H$ (however, see below).   Indeed, the identity (\ref{mainthm2.1}) directly links a weighted $L^2$-integral of  the holomorphic eigenfunction continuations over $H$  to boundary QER. That part of the argument is somewhat technical and uses the curvature assumption on $H$ (see sections \ref{guts} and \ref{opanalysis}).

Despite the fact that the growth of the holomorphic continuations $u_h^{H,\C}$ and consequently, the nodal intersection count, need not be directly linked to the QER problem for {\em Dirichlet} data consisting of eigenfunction restrictions to $H$, it is worthwhile to point out that it is directly related to the QER problem for {\em Cauchy} data along $H.$  Let
$$CD_H(h):= (\phi_{h}^{H}, \phi_{h}^{H,\nu})$$
with $\phi_h^H = \phi_h |_H $ and $\phi_{h}^{H,\nu} = h \partial_{\nu_H} \phi_h |_H,$ where we continue to assume that $(\phi_h)$ is QE sequence for the domain $\Omega.$ Then, the interior curve $H$ and $\partial \Omega$ bound a subdomain $\Omega_H \subset \Omega$ and one can write the  boundary restriction $\phi_h^{\partial \Omega}$ directly in terms of the Cauchy data along $H.$  Indeed, Greens formula gives
\begin{align} \label{cd1}
\phi_{h}^{\partial \Omega}(r(s)) = \int_{H} \partial_{\nu_H(q(t)} G_0(r(s),q(t),h) \phi_h^H(q(t)) dq(t) - \int_{H} h^{-1}G_0(r(s),q(t),h) \phi_h^{H,\nu}(q(t)) dq(t). \end{align}
It follows that 
\begin{align} \label{cd2}
\langle P(h) \phi_h^{\partial \Omega}, \phi_h^{\partial \Omega} \rangle = \langle {\bf Q}_H(h) \, CD_H(h), CD_H(h) \rangle \end{align}
where ${\bf Q}_H(h)$ is a $2\times 2$-matrix of $h$-pseudodifferential operators acting on $H$ and we write $CD_H(h)$ as a column vector. Consider the operators $T(h): C^{\infty}(H) \to C^{\infty}(\partial \Omega)$ and $G(h):C^{\infty}(H) \to C^{\infty}(\partial \Omega)$ with Schwartz kernels $T(r(s),q(t);h)= \partial_{\nu_H} G_0(r(s),q(t);h)$ and $G(r(s),q(t);h) =  -h^{-1} G_0(r(s),q(t);h)$ respectively. Here $G_0(x,y,h)$ is the free Greens kernel in (\ref{Hankel}) and  both $T(h)$ and $S(h)$ are $h$-Fourier integral operators with standard WKB-expansions. The latter follows from (\ref{Hankel}) since $H$ is interior to $\Omega$ and so dist$(H,\partial \Omega) >0.$ The entries of ${\bf Q}_{H}(h)$ are the operators $Q_{11}(h) = T(h)^*P(h) T(h), \, Q_{22}(h) = G(h)^* P(h) G(h),\, Q_{12}(h) = T(h)^* P(h) G(h)$ and $Q_{21}(h) = Q_{12}(h)^*.$ By the $h$-Egorov theorem, $Q_{ij}(h): C^{\infty}(H) \to C^{\infty}(H)$ are $h$-pseudodifferential on $H$ of order $0$ and the respective symbols can be computed in terms of the symbol $\sigma(P(h))$ in Proposition \ref{mainthm2prop} and the transfer map between $\partial \Omega$ and $H$ (see \cite{TZ2} section 3). One can then restate Proposition \ref{mainthm2prop} in the form
\begin{equation}\label{cdform}
h^{-1/2} \int \int_{\C/ 2\pi \Z} e^{-2S(t)/h} | u_h^{H,\C}(t)|^2 \chi_{\eo,\delta}(t) dt d\overline{t} \sim_{h \to 0^+} \langle {\bf Q}_H(h) \, CD_H(h), CD_H(h) \rangle.\end{equation}
The formula (\ref{cdform}) relates the growth of the holomorphic continuations $u_h^{H,\C}$ to QER for the eigenfunction Cauchy data along $H.$

\begin{rem} Recently, Zelditch \cite{Z} has obtained detailed results on the asymptotic distribution of complex zeros of $\phi_h^{H,\C}$ in the ergodic case when $H$ is a geodesic. Although we do not pursue this here, the identity in (\ref{mainthm2.3}) can be used to derive asymptotic distribution results for complex zeros of $\phi_h^{H,\C}$ in the case where $H$ has strictly positive geodesic curvature, but only in an annular subdomain of $H_{\eo}^{\C}$ away from the real curve $H$ (ie. on the support of the cutoff $\chi_{\eo}$). At the moment, we do not know what the asymptotic distribution of the zeros of $\phi_h^{H,\C}$ looks like in the entire Grauert tube $H_{\eo}^{\C}$ when $H$ is geodesically curved. We hope to return to this problem elsewhere. \end{rem}

\begin{rem} The convexity assumption on  $\partial \Omega$ in Theorems \ref{mainthm2} and \ref{mainthm3} can be relaxed somewhat. Although we do not pursue it here, it is not hard to show using the methods of this paper that our results extend to the case  where  the glancing map $Y: \partial \Omega \to H$ is a  diffeomorphism, $H$ is $C^{\omega}$ strictly convex and $(\phi_h)$ is a QE sequence.\end{rem}

We thank Peter Sarnak and Steve Zelditch for helpful comments regarding an earlier version of the manuscript.

\section{Analytic continuation of eigenfunctions and domains}

\subsection{Complexification of domains $\Omega$ and their boundaries $\partial \Omega$} \label{parametrization} We adopt notation that is similar to that of Garabedian \cite{G} and Millar \cite{M1,M2} (see also \cite{TZ}) and
 denote points in $\R^2$ by $(x,y)$ and  complex coordinates in $\C^2$ by $(z,w).$ It is also important to single out the independent complex coordinates $\zeta = z + iw,  \zeta^* = z - iw.$
 When $H \subset \Omega$ and $\partial \Omega$ are real analytic curves, their complexifications are the images of analytic continuations of real analytic parameterizations.
There are two natural parameter spaces and, as in \cite{TZ}, we
freely work with both throughout. 
We define the parameter strip of width $2\epsilon_0$  to be
  $$S_{\eo} = \{t \in  \C : t = \Re t+ i \Im t, \Re t\in \R , \Im t \in [-\eo, \eo ] \}.$$ 
  The corresponding fundamental rectangular domain is
  $$S_{\eo,\pi} = \{t \in  \C : t = \Re t+ i \Im t,  \Re t \in [-\pi,\pi] , \Im t \in [-\eo, \eo ] \}.$$  For $\eo >0$ small, the associated conformal map of $S_{\eo,\pi}$ onto $H_{\eo}^{\C}$ is
  $$\qc: S_{\eo,\pi} \longrightarrow H_{\eo}^{\C}$$
  $$\qc(t) = (q_1^{\C}(t) , q_2^{\C}(t)).$$ Without loss of generality, we assume that $H$ is a closed curve with $
 |q'(t)| \neq 0$ for all $t \in [-\pi,\pi].$
  In addition, we assume throughout that the real-analytic parametrization $q:[-\pi,\pi] \rightarrow H$ with $q(t + 2\pi) = q(t)$  extends to a conformal map $q^{\C}: S_{2\eo, 2\pi} \rightarrow H_{2\eo}^{\C}$ with $q^{\C}(t+2\pi) = q^{\C}(t).$
  One can also naturally parameterize $H_{\eo}^{\C}$ using functions on annular domains in $\C$ of the form
  $$A_{\eo}:= \{ z \in \C; e^{-\eo} \leq |z| \leq e^{\eo} \}.$$
  In terms of  the conformal map
  $$ z: S_{\eo,\pi} \longrightarrow A_{\eo}, \,\,
   z(t) = e^{it},$$
 given  any $2 \pi$-periodic holomorphic function $f \in O(S_{\eo,\pi})$ there is a unique holomorphic $F \in O(A_{\eo})$ with
  $$ f(t) = F(z(t)) = F(e^{it}).$$
 The conformal parametrizing map $\qc: S_{\eo,\pi} \rightarrow H_{\eo}^{\C}$ induces a conformal parametrizing map $Q^{\C}: A_{\eo} \rightarrow H_{\eo}^{\C}$ with
 $\qc(t) = Q^{\C}(e^{it})$. We use the two maps interchangeably throughout. Generally, upper case letters denote parametrization maps from the annulus $A_{\eo}$ and lower case ones denote maps from the rectangle $S_{\eo,\pi}.$
 In view of the potential layer formulas and the boundary conditions, the boundary curve $\partial \Omega$ has special significance. Without loss of generality, we let $r:[-\pi,\pi] \to \partial \Omega$  be the real analytic arclength  parametrization of the boundary with $r(t+2\pi) = r(t)$ and  $|r'(t)| = 1$ for all $t \in [-\pi,\pi].$  The corresponding holomorphic continuation is
 $r^{\C}: S_{\eo,\pi} \longrightarrow \partial \Omega_{\eo}^{\C}$
 with $r^{\C}(t) = R^{\C}(z(t)).$

In addition, we let $C_{\eo}$ be a simply-connected domain bounded by a closed real-analytic curve $\partial C_{\eo} $  with

 \begin{equation} \label{oval} \begin{array}{ll}
  [-\pi, \pi ] \subseteqq S_{\eo,\pi} \subseteqq C_{\eo} \subseteqq S_{2\eo,2\pi}, \end{array} \end{equation}
 and
$$ \min_{z \in \partial C_{\eo} \cap \R} | z - [-\pi,\pi] | \geq \frac{\pi}{2} \,\, \text{and} \,\, \max_{z \in \partial C_{\eo}} |\Im z| \leq \frac{7\eo}{4}.$$
 The interval $[-\pi,\pi]$ is just the totally real slice of the complex parameter rectangle $S_{\eo,\pi}$ which is contained in $C_{\eo}$. By possibly shrinking $\eo >0$ we assume from now on that the eigenfunction restrictions extend to $2\pi$-real periodic holomorphic functions $u_{h}^{H,\C}$ on the larger rectangles $S_{2\eo,2\pi}$.

\subsubsection{Holomorphic continuation of the restricted eigenfunctions} \label{potentials}

Let $ G: H^{-2}(\R^2) \rightarrow L^{2}(\R^2)$ be the fundamental solution of the Helmholtz equation in $\R^2$ with Schwartz kernel  $$G(x,y,x',y',h) = \frac{i}{4} \text{Ha}_{0}^{(1)}(h^{-1} | (x,y) - (x',y') |), $$ where
\begin{equation} \label{Hankel}
\text{Ha}_{\nu}^{(1)}(z) = c_{\nu}\frac{e^{iz}}{\sqrt{z}}\int\limits_0^\infty  {\frac{{e^{ - s} }}{{\sqrt s }}} (1 -\frac{s}{{2iz}})^{\nu  - \frac{1}{2}} ds, \,\,\,\, \Re z >0. \end{equation}
An application of Green's theorem yields the following potential layer
formula for the Neumann eigenfunctions
\begin{equation} \label{green}
\varphi_{h}(x,y) = \int\limits_{\partial\Omega}
\partial_{\nu_{s}}G(x,y;r(s),h)  \, \varphi_{h}(r(s))d\sigma(s),\end{equation}
where $(x,y) \in \mathring{\Omega}$ and
 $\nu_{s} \in S_{\partial \Omega}(\Omega)$ is the unit
external normal to the boundary at $r(s)\in \partial \Omega.$
We denote the kernel of the potential layer operator in (\ref{green}) by
\begin{equation} \label{N kernel} N(x,y;r(s),h) := \partial_{\nu_{s}}G(x,y;r(s),h) = - h^{-1} \text{Ha}_{1}^{(1)}(h^{-1} | (x,y)- r(s) |) \cos \theta( (x,y),r(s) ) \end{equation}
where
$$ \cos \theta( (x,y),r(s) ) = \left\langle \frac{ (x,y)-r(s) }{ |(x,y)-r(s)|},\nu_s \right\rangle$$
and the corresponding operator by $N(h):C^{\infty}(\partial \Omega) \rightarrow C^{\infty}(\mathring{\Omega}).$

To understand holomorphic continuation of eigenfunctions, one starts with the singularity decomposition of the kernel
$G(x,y;r(s),h)$. It is well-known that
\begin{equation} \label{pot1}
 G(x,y;r(s),h) = A(h^{-1}| (x,y) -r(s)|) \log \left( \frac{1}{| (x,y) -r(s)|} \right) + B(h^{-1}| (x,y) -r(s)| ) \end{equation} where
$A(z)$ and $B(z)$ are entire functions of $z^2 \in \C$ and each of them have
elementary expressions in terms of Bessel functions (see \cite{TZ} appendix
A). $A(z)$ is  the Riemann function \cite{G}.

\noindent We identify $(x,y) \in \mathbb{R}^2$ with $x + i y \in \mathbb{C}$, and introduce the notation  $\rho(x+iy,r(t)) = \sqrt{(x+iy-r(t))\cdot (x-iy- \overline{r(t)})}$ where $z \mapsto \sqrt{z}$ is the positive  square-root function with $\sqrt{\Re z} >0$ when $\Re  z >0.$
Substitution of (\ref{pot1}) in (\ref{green}) implies that for $(x,y) \in \mathring{\Omega}$ and with $\partial_{\nu}:= \partial_{\nu_s},$
\begin{align} \label{pot2}
\phi_h(x,y) =& -\frac{1}{2}\int_{\partial\Omega} \phi_h(r(s)) \partial_{\nu}A(h^{-1} \rho) \log  (\rho^2)  \, dr(s) \notag \\
        & - \frac{1}{2} \int_{\partial\Omega}  \phi_h(r(s))A( h^{-1} \rho) \partial_{\nu}\log(\rho^2) \, dr(s) + \int_{\partial\Omega}  \partial_{\nu}B(h^{-1} \rho) \phi_h(r(s)) \, dr(s).
\end{align}
The holomorphic continuation of the third integral is the easiest to describe  since there is a real analytic $F \in C^{\omega}(\R,\R)$ with entire extension  $F^{\C} \in O(\C)$ satisfying
\begin{equation} \label{entire}
 \partial_{\nu} B(h^{-1} \rho ) =  \partial_{\nu} \,  F(h^{-2}\rho^2) \end{equation} and the same is true for the normal derivative $\partial_{\nu} A(h^{-1} \rho)$ of the Riemann function.
In view of (\ref{entire}), the last integral in (\ref{pot2}) has a biholomorphic  extension to $\Omega^{\C} := \{ (z,w) \in \C^2; \Re z + i \Re w \in \Omega \}.$

In contrast, the first two integrals both turn out to have fairly subtle analytic continuations over $\Omega$ in $\C^2$ that  rely heavily on analytic continuation of the eigenfunction boundary traces (\cite{TZ} Appendix 9). However, we need only consider holomorphic continuation over a strictly interior curve $H \subset \mathring{\Omega}$ here.
Thus, to describe the holomorphic continuation of the first integral on the right hand side of (\ref{pot2}) it suffices to assume  that $x+iy \in \mathring{\Omega}$ is far from the boundary with    $ |x+iy-r |^2 > 4\eo^2 >0,$ where
$ \eo < \text{dist} (H,\partial \Omega).$ When $\max( |\Im w|, |\Im z|) < \eo,$ it follows by Taylor expansion that
$$ | \rho^2(z+iw,r) - |\Re z +i \Re w- r|^2  \,  | \leq \max (|\Im w|, |\Im z|) | | \Re z +i  \Re w - r|.$$
Thus, $\Re \rho^2(z+iw,r) > \eo^2$ and
 $ (s,\Re z, \Re w) \mapsto \log (\rho^2(\Re z,\Re w,s) ) $ has a biholomorphic continuation in the $(\Re z, \Re w)$ variables to
 \begin{equation} \label{contregion}
  [ \Omega - \partial\Omega_{2\eo}]^{\C}(\eo) =\{ (z,w) \in \Omega^{\C}; \min\limits_{r \in \partial \Omega} |\Re z + i \Re w - r| \geq 2 \eo, \, |\Im z| \leq \eo, |\Im w| \leq \eo \}.\end{equation}
 The same is true for $\partial_{\nu} A(\rho)$ and consequently, for the integral.   By the same argument, the funcrtion $(s,\Re z, \Re w) \mapsto \partial_{\nu_s} \rho^2(\Re z, \Re w, s)$ also biholomorphically continues  in the $(\Re z, \Re w)$-variables to $ [ \Omega - \partial\Omega_{2\eo}]^{\C}(\eo).$ Consequently, so does the second integral on the RHS of (\ref{pot2}).

Restriction of the outgoing  variables, $ (x,y)$ to $(q_{1}(s),q_2(s)) \in H$ in (\ref{pot2})
  yields the integral equation
\begin{equation} \label{FONE}
N(h) \phi_{h}^{\partial \Omega} = \phi_{h}^H. \,\, \end{equation}
From now on, we will refer to $\eo >0$ as the {\em modulus of
analyticity}. In light of the potential layer formula (\ref{FONE}) for the Neumann
eigenfunctions, it is useful to compare eigenfunction restrictions
to $\partial \Omega$ with restrictions to $H \subset
\mathring{\Omega}$  and similarly, for the holomorphic
continuations.  
 For the
restrictions of the Neumann eigenfunctions pulled-back to the parameter domain, we continue to write
 \begin{equation} \label{notation1} \begin{array}{ll}
 u_{h}^{\partial \Omega}(t)= \phi_{h}^{\partial \Omega} ( r_1(t), r_2(t)), \,\,\, u_h^H(t) = \phi_h^H (q_1(t), q_2(t)); \,\,\, t \in [-\pi,\pi]\end{array} \end{equation}
with $r(t) = r_1(t) + i r_2(t) \in \partial \Omega$ and $q(t) = q_1(t) + i q_2(t) \in H$.

\begin{prop} \label{N} Suppose that $H \subset \Omega$ is  real analytic and let
$dist (H, \partial \Omega) = \min_{(s,t)\in [-\pi,\pi]^{2}} |q(t)-r(s)|.$
Assume that $q(t)$ has a holomorphic continuation to  $I(\delta) = [-\pi,\pi] \pm  [-\delta, \delta ].$
 Then the restriction $u_{h}^H$  of the Neumann eigenfunctions has a holomorphic continuation $u_{h}^{H,\C}(t)$ to the strip
 $S_{2\eo,2\pi} $ with
 $$2 \eo < \frac{ dist (H, \partial \Omega) }{ \sup_{t \in I_{H}(\delta) } |\partial_t q^{\C}(t)| }.$$ Moreover, in the strip $S_{2\eo,2\pi},$ the continuation is given by complexified potential layer equation
 \begin{equation} \label{cpl}
 N^{\C}(h) \phi_h^{\partial \Omega} =  \phi_h^{H,\C}, \end{equation}
 where $N^{\C}(h)$ is the operator with Schwartz kernel $N^{\C}(q^{\C}(t),r(s),h)$ holomorphically continued in the outgoing $t$-variables, and $\phi_h^{H,\C}$ the holomorphic continuation of $\phi_{h}^{H}$ to $H_{\eo}^{\C}.$
\end{prop}

\begin{proof} The proposition follows from the above analytic continuation argument for (\ref{pot2}) and (\ref{FONE}) since  by (\ref{contregion}) the $u_h^H$ holomorphically continue to the set $\{ t \in \C; \min_{r \in \partial \Omega}  |q(\Re t) - r | \geq 2\eo, | \Im q^{\C}(t)| < \eo \}.$  The  formula in (\ref{cpl}) follows from uniqueness of analytic continuation and the fact that, by the  above analysis of  (\ref{pot2}), for $ \zeta = q^{\C}(t) \in H_{2\eo,2\pi},$
$$\phi_{h}^{H,\C}(\zeta) = [ N(h) \phi_h^{\partial \Omega} ]^{\C}(\zeta) = N^{\C}(h) \phi_h^{\partial \Omega}(\zeta).$$ \end{proof}

\section {The frequency function and measure of the nodal set}
We first recall the definition of the  frequency function with an
important application due to F.H. Lin \cite{L} for estimating
measures of nodal sets. We are interested here in the planar case of holomorphic functions. In general, the frequency function
for harmonic functions in arbitrary dimensions is defined as
follows
\begin{defn} \label{ff}
 Let $\Delta u = 0$ with $\Delta = \sum_{j=1}^{n} \partial_{x_j}^2$ the standard Laplacian in ${\mathbb R}^n$. The frequency function of the harmonic function $u$ in the unit ball $ B_{1} \subset \mathbb{R}^n$ is defined to be
$$ F(u) = \frac{{\iint_{B_1 } {|\nabla u|^2 } }}{{\int\limits_{\partial B_1 } {|u|^2 } }}.$$
\end{defn}

When the context is clear, we suppress the dependence of $F$ on
$u$ and just write $F$ for the frequency function. In the planar
case, any non-zero  holomorphic function  $f(z)$ in the disc $B_1 = \{z \in
{\mathbb C}; |z| \leq 1 \},$  has a decomposition  of the form $f
= u + i v$ where $u,v$ are harmonic conjugates and so, since  $ \partial_{z} f =  \partial_{x} u + i \partial_x v,$   in analogy with the harmonic
case in  Definition \ref{ff}, one defines the frequency function
to be
 \begin{equation} \label{ffhol}
 F = \frac{{  \iint_{B_1 } {|\partial_z f(z)|^2 } dz d\bar{z} }}{{\int\limits_{\partial B_1 } {|f(z)|^2 } d\sigma(z) }}. \end{equation}

An elementary but useful example to keep in mind is the monomial $f(z) = z^k = r^k e^{ki\theta}; \, k \in {\mathbb Z}^+.$ In this case, one easily computes the frequency
function to be $   k^2 \int_{0}^1 r^{2k-1} dr = k/2,$ where $k$ is the degree
of the polynomial $z^k.$  By Green's formula, the analogous result
is easily verified for arbitrary homogeneous harmonic polynomials
in any dimension. The following result, proved by Lin \cite{L}
using Taylor expansion, and by Han \cite{H} using Rouche's theorem
is an important generalization of the polynomial case to arbitrary
non-zero holomorphic functions. We recall the result here and
refer the reader to \cite{H}  for a proof (see also the upcoming
book of Han and Lin \cite{HL}). The key result that estimates the
number of complex zeros of $f(z)$ in the disc $B_1$ is given by

\begin{theo}\cite{H,L,HL} \label{t2.2}
Let $f(z)$ be a non-zero analytic function in $B_1= \{z \in
\mathbb{C} : |z|\leq 1 \}.$ Then, for some universal $\delta \in (0,1),$ $$\#
\{f^{-1}(0)\cap B_{\delta}\}\leq 2 F,$$
 where $F$ is defined to be  the ratio in (\ref{ffhol}). \end{theo}


 It is useful here to rewrite the frequency function $F$  in (\ref{ffhol}) exclusively in terms of integrals over the circular disc boundary $\partial B_1.$


\begin{lem} \label{bdy lemma}  Let $f: B_1\rightarrow \C$ be non-zero holomorphic. Then,
 $$ F \leq  \frac{\|\partial_{\theta}   f \|_{L^2(\partial B_1)}}{\|  f \|_{L^2(\partial B_1)}} ,$$ where $\partial_{\theta} = x \partial_y - y \partial_x$ is the unit tangential derivative along the circular boundary $\partial B_1$ of the disc.
\end{lem}

\begin{proof}
The proof follows from Green's formula and an application of
Cauchy-Schwarz. For $z =x +iy = (x,y) \in B_1$ we write $f(z)
= \Re f(x,y ) + i \Im f(x,y),$ where $ \Re f (x,y ) , \Im f(x,y)$
are real-valued harmonic functions.

Since $f$ is analytic,
 $ \partial_z f = \partial_x \Re f - i \partial_y \Re f,$
and so, $$ |\partial_z f |^{2} =  (\partial_x \Re
f)^{2} + (\partial_y \Re f)^{2} =  |\nabla (\Re
f)|^{2}.$$ An application of Green's theorem implies that
\begin{equation} \label{greenformula}
\begin{array}{lll}
 \iint\limits_{B_1 }|\partial_z f(z) |^{2} dz d\overline{z}  & =  \iint\limits_{B_1 }|\nabla (\Re f)|^{2} \, dx dy \\ \\
& =  \int\limits_{\partial B_1} \Re f  \cdot  \partial_{\nu} (\Re f) \, d\theta -  \iint\limits_{B_1 }\Re f \cdot \Delta (\Re f) dxdy \\ \\
& =  \int\limits_{\partial B_1} \Re f \cdot
\partial_{\nu}(\Re f)\, d\theta,
\end{array}
\end{equation}
where, $\nu$ is the outward pointing unit normal to $\partial B_1$
and the last line follows since $\Delta (\Re f) = 0$ in $B_1.$
\newline

Next, we use the  Cauchy-Riemann equations written in polar
coordinates $(r,\theta) \in \R^+ \times [0,2\pi)$ to rewrite the
normal derivative term on the right hand side of the last line in
(\ref{greenformula}) in terms of a tangential one.
\begin{equation} \label{cr polar}
 \partial_{\nu} \Re f |_{\partial B_1}  = \partial_r \Re f |_{r=1}= \partial_{\theta} \Im f |_{r=1}. \end{equation}

Hence, it follows from (\ref{cr polar}) and  (\ref{greenformula})
that

\begin{equation} \label{upshot1}
\iint\limits_{B_1 }|\partial_z f(z) |^{2} dz d\overline{z}  =\int\limits_{\partial B_1} \Re f \cdot
\partial_{\theta} (\Im f) \, d\theta. \end{equation}

Finally, an application of  Cauchy-Schwarz in (\ref{upshot1})
gives

\begin{equation} \label{cs} \begin{array}{lll}
 \iint\limits_{B_1} | \partial_z f(z) |^{2} dz d\overline{z} & \leq \|\Re f\|_{L^2(\partial B_1)}  \cdot  \|\partial_\theta (\Im f)\|_{L^2(\partial B_1)}  \\ \\
& \leq \| f \|_{L^2(\partial B_1)}  \cdot
\|\partial_\theta f \|_{L^2(\partial B_1)}.
\end{array} \end{equation}

\end{proof}


\subsubsection{Frequency functions for the holomorphic continuations of restricted eigenfunctions}

We wish  to estimate here the intersection number  $n(h,H)$
in terms of  Lemma \ref{bdy lemma}.

\begin{prop} \label{strip bound}
Let $H \subset \mathring{\Omega}$ be a  $C^{\omega}$ interior curve and $C_{\eo}$ be a simply-connected, bounded  domain in $\C$ containing the
rectangle $S_{\eo,\pi}$ with   real-analytic  boundary $\partial C_{\eo}$ and arclength parametrization $t \mapsto \kappa(t) \in \partial C_{\eo}.$   Then, for $\eo >0$ sufficiently small

\[ n(h,H) \leq C_{H,\eo}  \frac{ \| \partial_{T}  u_{h}^{H,\C} \|_{L^2_{\eo} } }{ \| u_{h}^{H,\C} \|_{ L^2_{\eo} } }. \]

Here, $L^{2}_{\eo} := L^2 (\partial C_{\eo}, |dt|)$ and  $\partial_{T}$ denotes the unit tangential derivative along $\partial C_{\eo}$ with $\partial_{T} f(t):= \frac{d}{dt} f(\kappa(t)).$  \end{prop}

\begin{proof}

Since $C_{\eo}$ is a simply-connected bounded domain,  by the Riemann mapping theorem  there exists a conformal map
 \[ \kappa: \mathring{B_1} \rightarrow C_{\eo} ,\] where $\mathring{B_1} = \{ z; |z| <1\}$.
By Caratheodory, there is $\tilde{\kappa} \in C^{0}( B_1)$  with $\tilde{\kappa}|_{\mathring{B_1}} = \kappa|_{\mathring{B_1}}$ univalent up to the boundary. Moreover,  since $\partial C_{\eo}$ is real-analytic, it follows from the Schwarz reflection principle that
\begin{equation} \label{nice}
\tilde{\kappa} \in C^{\omega}( B_1).
\end{equation}
Analogous results also hold for the inverse conformal map $\kappa^{-1}: C_{\eo} \rightarrow \mathring{B_1}.$
Since $\kappa$ is conformal and  satisfies (\ref{nice}), it follows that the boundary restriction
\[ \tilde{\kappa}|_{\partial B_1}:   \partial B_1 \rightarrow \partial C_{\eo} \]
is a $C^{\omega}$-diffeomorphism.
We define the composite function on $B_1$
\[g_{h}^{H,\C}(z): = u_{h}^{H,\C}(\tilde{\kappa}(z)); \,\,\, z \in B_1. \]
We apply theorem \ref{t2.2} to the holomorphic function $g_{h}^{H,\C}$ in $B_1$. We choose $\delta \in (0,1)$ so that $C_{\delta} := \tilde{\kappa}(B_\delta) \supset [-\pi,\pi]$.
We have that
\begin{equation} \label{real}
n(h,H) = N_{u_{h}}\cap [-\pi,\pi] \leq n^{\C}(h,C_{\delta}) =  \# \{ t \in C_{\delta};   u_{h}^{H,\C}(t) = 0 \} = \# \{ t \in B_{\delta};   g_{h}^{H,\C}(t) = 0 \}.
\end{equation}
It follows by Theorem \ref{t2.2}, Lemma \ref{bdy lemma} and(\ref{real}) that
\begin{equation} \label{ratio}
n(h,H) \leq  2\frac{\|\partial_{\theta}   g_{h}^{H,\C}\|_{L^2(\partial B_1)}}{\|  g_{h}^{H,\C} \|_{L^2(\partial B_1)}}. \end{equation}  An application of the change of variables formula in (\ref{ratio})   with  $t = \tilde{\kappa}(z) $ for $z \in \partial B_1 $ proves the proposition.
\end{proof}

\section{Estimating the frequency function: $h$-microlocal decomposition}

In view of Proposition \ref{strip bound}, we are left with showing
that
\begin{equation} \label{guts}
  \frac{ \| \partial_{T}  u_{h}^{H,\C} \|_{L^2_{\eo} } }{ \| u_{h}^{H,\C} \|_{ L^2_{\eo} }  } = {\mathcal O}_{\Omega,H} (h^{-1}). \end{equation}
  To prove (\ref{guts}), we will need to h-microlocally decompose $ \gamma_{\partial C_{\eo}} u_{h}^{H,\C}$ where $\gamma_{\partial C_{\eo}}:C^{0}(S_{2\eo,2\pi}) \rightarrow  C^{0}(\partial C_{\eo})$  is the restriction map. We briefly digress here to introduce the relevant h-pseudodifferential cutoff operators noting that  $\partial C_{\eo}$ is $C^{\omega}$-diffeomorphic to the unit circle $\partial B_{1}.$

\subsection{Semiclassical pseudodifferential operators on tori}

Let $M^n$ be  compact manifold. The following semiclassical symbol
spaces are standard \cite{EZ} and will suffice for our purposes.

\begin{defn} \label{symbol}
We say that $a \in S^{k,m}_{cl}(T^*M \times [0,h_0))$ if $a \in
C^{\infty}(T^*M;[0,h_0))$ has an asymptotic expansion of the form
$a \sim_{h \rightarrow 0^+} h^{-k}\sum_{j=0}^{\infty} a_{j}(x,\xi)
h^{j}$ where
$$ | \partial_{x}^{\alpha} \partial_{\xi}^{\beta} a_{j}(x,\xi)| \leq C_{\alpha,\beta} (1 + |\xi|)^{m - |\beta|}; \,\, (x,\xi) \in T^*M.$$
The corresponding class of h-pseudodifferential
 operators $A_h: C^{\infty}(M) \rightarrow
C^{\infty}(M)$ have Schwartz
kernels locally of the form
$$A_h(x,y) = (2\pi h)^{-n} \int_{\R^{n}} e^{i\langle x-y,\xi \rangle/h} a(x,\xi;h) d\xi$$ with
$a \in S^{k,m}_{cl}(T^*M;[0,h_0))$. We write $A_h = Op_h(a)$ for
the operator with symbol $a(x,\xi;h).$
\end{defn}

Since $\partial C_{\eo}$ is $C^{\omega}$-diffeomorphic to a circle
$S^1 = \R/ 2\pi \Z,$ it suffices here to consider h-pseudodifferential
operators on tori and the latter  operators can  be conveniently
described globally in terms of their action on Fourier
coefficients. Given $A_h \in Op_h(S^{0,m}(T^* {\mathbb T}^n))$ one
can  write the Schwartz kernel in the form
$$ A_h(x,y) = (2\pi )^{-n} \sum_{\xi \in (h\Z)^n} e^{i \langle x - y, \xi \rangle /h } a_{\T}(x,\xi;h); \,\, (x,y) \in [-\pi,\pi]^n \times [-\pi,\pi]^n $$
where $a_{\T}(\cdot,\xi) \in C^{\infty}({\T})$ and
$$ | \partial_{x}^{\alpha} \Delta^{\beta}_{h,\xi} a_{\T}(x,\xi) | \leq  C_{\alpha,\beta} ( 1 + |\xi|)^{m- |\beta| }$$
where $\Delta^{\beta}_{h,\xi}  a_{\T}(x;\xi_1,...,\xi_n)= a_{\T}(x;\xi_1+h\beta_1,...,\xi_n + h\beta_n) -
a_{\T}(x;\xi_1,...,\xi_n)$ is the semiclassical iterated difference
operator in the frequency coordinates. The converse also holds, so
that the two realizations of h-pseudodifferential operators are
equivalent  (see \cite{Ag,Mc} for the homogeneous case where $h =1.$ The extension to the semiclassical setting is straightforward).

We are interested here specifically  in the  h-pseudodifferential
cutoffs $\chi_h = Op_h(\chi) \in Op_h(S^{0,-\infty}(T^*\partial C_{\eo}))$ where $\chi \in C^{\infty}_{0}(T^*
\partial C_{\eo} )$. We naturally identify $\partial C_{\eo}$ with
$\R/ 2\pi \Z$ by using the periodic $C^{\omega}$ arclength
parametrization
$$ \kappa: [-\pi,\pi] \rightarrow \partial C_{\eo}; \, \, t \mapsto \kappa(t).$$

\subsection{Semiclassical wave front sets of eigenfunction restrictions}
Let $H^{n-1} \subset M^n$ be any interior {\em smooth}
hypersurface in a compact manifold with or without boundary. In
this subsection, we do not make any analyticity assumptions on
either $H$ or the ambient manifold, $M.$  Let $u_h^H := \gamma_H
\phi_h$ be the eigenfunction restriction where $\gamma_H: f
\mapsto f|_H, f \in C^0(H).$ Then, making a Fermi-coordinate decomposition in a collar neighbourhood of $H,$ it is not hard to show that
\begin{equation} \label{wf1}
WF_h(u_h^H) \subset B^*H = \{ (s,\sigma) \in T^*H; |\sigma|_g \leq
1 \}. \end{equation}
 For Euclidean
domains $M = \Omega$,  (\ref{wf1}) follows directly from potential layer
formulas. For completeness and because of the  importance of the localization of $WF_h(u_h^H)$ in our argument,  we sketch the proof of (\ref{wf1}) for planar domains, which is the case we are interested in here.
The proof of (\ref{wf1}) uses the potential layer representations of eigenfunctions discussed in
Subsection \ref{potentials} in  the planar case $n=2$  restricted to
the curve $H$.   It is immediate from (\ref{green})
that

\begin{equation} \label{pl}
u_{h}^{H}(t) = \int_{-\pi}^{\pi} N(q(t),r(s);h) \,\, u_h^{\partial
\Omega}(s)  d\sigma(s). \end{equation}

Since $H \subset \Omega$ is interior, $\inf\limits_{t,s\in [-\pi,\pi]}
|q(t) - r(s)| \geq C >0$ and so, from (\ref{Hankel}) it follows that
\begin{equation} \label{asymptotics}
\tilde{N}(t,s;h) := N(q(t),r(s),h)  = (2\pi h)^{-\frac{1}{2}}
e^{ih^{-1} |q(t) - r(s)|} a(t,s;h) \end{equation} where,
$$a(t,s;h) = \sum\limits_{j=0}^{k} a_{j}(t,s) h^{ j} + {\mathcal O}(h^{k+1}) $$
uniformly for all $(q(t),r(s)) \in H \times \partial \Omega$ with
$a_{j} \in C^{\infty}([-\pi,\pi] \times [-\pi,\pi]).$   Similar
uniform asymptotics hold for derivatives as well.


Let $\chi(\xi) \in  C_{0}^{\infty}(\R)$ be a cut-off function
equal to zero when $|\xi| \geq 2$ and equal to $1$ for $|\xi|<
3/2$ and let $Op_h (\chi) \in Op_{h}(S^{0,-\infty}(T^*H;(0,h_0]) )$ be the
microlocal cut-off with kernel
$$Op_h(\chi)(t,t') = (2\pi )^{-2} \sum_{\xi \in h\Z} e^{i \langle  t - t', \xi \rangle/h} \, \chi(\xi); \,\, (t,t')  \in [-\pi,\pi] \times [-\pi,\pi]. $$
Then, from (\ref{pl}) and (\ref{asymptotics}), it follows that
\begin{align*}
&Op_{\h} (1-\chi) u_h^H (t) \\
&= Op_h(1-\chi )  N u_h^{\partial \Omega} (t)\\
 &=(2\pi )^{-2} \sum_{\xi \in h\Z}  \int_{-\pi}^{\pi}\int_{-\pi}^{\pi}e^{i [ (t-t') \xi + |q(t') - r(s)| ]/h}  \, (1-\chi)(\xi) \, a(q(t'),r(s);h)  \, u_{h}^{\partial \Omega}(s) \, ds \, dt.'
 \end{align*}
Since $|d_{t'}q(t')| = 1,$ differentiation of the phase
$$ \Psi(t,t',s;\xi): = (t-t') \xi + |q(t')-r(s)|$$
in $t'$ gives

\[  | \partial_{t'} \Psi(t,t',s;\xi) | =  \Big| - \xi + \left\langle d_{t'}q(t'), \frac{ q(t') - r(s)}{|q(t')-r(s)|} \right\rangle \Big| \geq  |\xi| - 1 \geq \frac{1}{2}; \,\, \text{when} \,\, |\xi| \geq \frac{3}{2}. \]
Since $|\xi| \geq \frac{3}{2}$ when $\xi  \in  \, $supp$ \chi,$
repeated integration by parts in $t'$, an application of
Cauchy-Schwarz and using that $\| u_{h}^{\partial \Omega} \|_{L^2}
= {\mathcal O}(h^{-1/4}) $ \cite{BGT} implies that $\sup_{t \in
[0,2\pi]} | Op_{\h} (1-\chi(\xi)) u_h^H(t)| = {\mathcal
O}(h^{\infty} \langle \xi \rangle^{-\infty})$ where $\langle \xi \rangle:= \sqrt{1 + |\xi|^2}.$ The same argument for $t$-derivatives combined with the Sobolev lemma
 implies that for all $k \in \Z^+,$
\begin{equation} \label{wfupshot1}
\|  Op_{\h} (1-\chi(\xi)) u_h^H \|_{ C^{k}([-\pi,\pi]) } = {\mathcal
O}_{k}(h^{\infty} \langle \xi \rangle^{-\infty}). \end{equation} The wavefront bound in
(\ref{wf1}) is an immediate consequence of (\ref{wfupshot1}) since
the cutoff function $\chi(\xi)$ can be chosen with support
arbitrarily close to $|\xi|=1$ and the same argument gives
(\ref{wfupshot1}) for any such cutoff.

In the next section we improve the compactness result (\ref{wf1}) under the  real-analyticity assumption on $(\partial \Omega, H)$ to show that in the h-microlocal decomposition (\ref{outline2}) the residual term $\| \partial_{T} (1-Op_h(\chi_R)) u^{H,\C}_h \|_{\eo}^2 = {\mathcal O}(e^{-C_0 \langle R \rangle  /h})$ with appropriate $C_0 >0$ and where  $\chi_R \in C^{\infty}_{0}(\R)$ with supp $\chi_{R} \subset \{ \xi; |\xi| \leq R \}.$  Hence, to get an asymptotic estimate for the frequency function of $u_{h}^{H,\C}$, it suffices to bound $\| \partial_{T} Op_h(\chi_R) u^{H,\C}_h\|_{\eo}$ and the latter is ${\mathcal O}(h^{-1} \| u^{H,\C} \|_{\eo} )$ by standard $L^2$-boundedness of the h-pseudodifferential operator $h \partial_T \chi_h \in Op_h( S^{0,-\infty}(T^*\partial C_{\eo})).$

\subsection{The real  analytic case} We now assume that $H$ is real-analytic. As outlined in the previous section,  our goal here is to improve the ${\mathcal O}(h^{\infty})$-bound in (\ref{wfupshot1}) to obtain exponential decay estimates for the residual mass term of the form $\| Op_{h}(1-\chi) u_{h}^{H,\C} \|_{L^2_{\eo}} = {\mathcal O}(e^{-C_0/h}).$   In the following, using the parametrization $[-\pi,\pi] \ni t \mapsto \kappa(t)$, we identify $\partial C_{\eo}$ with $\R/(2\pi \Z)$  and so, $Op_h(1-\chi): C^{\infty}(\R /2\pi \Z) \rightarrow C^{\infty}(\R /2\pi \Z).$

\subsubsection{Holomorphic continuation of the $\tilde{N}(t,s;h)$-kernel.}

Given $(z,w) \in \C^2,$  consider the map $ z+iw \mapsto
(z+iw)^* = z - iw$ which is the holomorphic continuation to $\C^2$
of the usual complex conjugation $x+iy \mapsto x-iy$ when $(x,y)
\in \R^2.$  In the following,  $z \mapsto z^{1/2}$ denotes the
 square root with positive real part with $-\pi < arg(z) \leq \pi.$

In view of Proposition \ref{N} it follows that for $\eo >0$ sufficiently small, the potential layer equation $u_{h}^{H}(t) = N
u_{h}^{\partial \Omega}(t)$ analytically continues to the equation
\begin{equation} \label{holcont1}
u_{h}^{H,\C}(\zeta) = [N u_{h}^{\partial \Omega}]^{\C}(\zeta); \,\,\, \zeta
\in S_{2\eo,2\pi}. \end{equation}
In particular, we consider here the case where $ \zeta = \kappa(t) \in \partial C_{\eo}$. 

For  $\zeta \in U_{\eo}$, where $U_{\eo} := \{\zeta \in S_{2\eo,2\pi}; \max\limits_{z\in \partial C_{\eo} } |z-\zeta| < \frac{\eo}{2}\}$, equation (\ref {holcont1}) remains valid and moreover, since
\begin{equation} \label{holphase}
\Re \,  [ q^{\C}(\zeta) - r(s)] [q^{\C}(\zeta)^* - \overline{r(s)}] \gtrapprox  \eo^2 >0 \,\, \text{when} \,\,  (\zeta,s) \in  U_{\eo}\times [-\pi,\pi],\end{equation}
 the kernel
\begin{equation} \label{holcont2}
N^{\C}(q^{\C}(\zeta),r(s),h) = \text{Ha}^{(1)}_{1} \left( h^{-1} \sqrt{ [ q^{\C} (\zeta)- r(s))][ q^{\C} (\zeta)^* - \overline{r(s)})]}\right)
 \end{equation}
 is holomorphic for
 $ \zeta \in U_{\eo}.$ By Proposition \ref{N}, we have

\begin{equation} \label{holcont3}
u_{h}^{H,\C}(\zeta) = \int_{-\pi}^{\pi}  N^{\C}(q^{\C}(\zeta),r(s),h) \, u_{h}^{\partial \Omega}(s) d\sigma(s),\,\,\, \zeta \in  U_{\eo}. \end{equation}

It follows from  (\ref{holphase}),  (\ref{holcont2})  and the integral formula (\ref{Hankel}) that the real WKB asymptotics for the $N(t,s,h)$-kernel  \cite{HZ,TZ} holomorphically continues in $t$ to give the complex asymptotic formula
\begin{equation} \label{complex wkb}
N^{\C}(q^{\C}(\zeta),r(s),h)  = (2\pi h)^{-\frac{1}{2}} e^{ i\rho^{\C}(q^{\C}(\zeta), r(s))/h} a^{\C}(\zeta,s;h); \,\, (\zeta,s) \in
U_{\eo}\times [-\pi,\pi], \end{equation} where,
$a^{\C}(\zeta,s;h) \sim_{h \rightarrow 0} \sum_{k=0}^{\infty}
a_{k}^{\C}(\zeta,s) h^{k}$ with $a_{k}(\cdot,s) \in {\mathcal O}(
U_{\eo})$ and
\begin{equation} \label{complexified phase}
\rho^{\C}(q^{\C}(\zeta),r(s)) = \sqrt{ [ q^{\C}(\zeta)- r(s)] [ q^{\C}(\zeta)^* - \overline{r(s)}] }; \,\, (\zeta,s) \in U_{\eo} \times [-\pi,\pi]. \end{equation}
In particular, for $\zeta =\kappa(t) \in \partial C_{\eo},$  we have
\begin{equation}\label{holcont4}
u_{h}^{H,\C}(\kappa(t)) = \int_{-\pi}^{\pi} N^{\C}(q^{\C}(\kappa(t)),r(s),h) \, u_{h}^{\partial \Omega}(s) d\sigma(s),\,\,\, t\in [-\pi,\pi],
\end{equation}
where $N^{\C}(q^{\C}(\kappa(t)),r(s),h)$ satisfies the asymptotics in  (\ref{complex wkb}). Since we compute in the parametrization variables $(t,s) \in [-\pi,\pi],$ to simplify notation we define
\begin{equation} \label{tildekernel}
\tilde{N}^{\C}(t,s,h):= N^{\C}(q^{\C}(\kappa(t)),r(s),h); \,\,\, (t,s) \in [-\pi,\pi] \times [-\pi,\pi]. \end{equation}




\subsubsection{ Estimating the residual kernel $[Op_h(1-\chi) \tilde{N}^{\C}] (t,s;h)$}
 Let $\chi \in C^{\infty}_{0}(\R)$ be a cutoff with $\chi(\xi) = 0$ when

$$ |\xi| \geq  20 \eo^{-1} \, \sup_{(\zeta,s) \in U_{\eo} \times [-\pi,\pi]}  | \rho^{\C}(q^{\C}(\zeta),r(s) ) |$$ and
 $\chi(\xi) = 1$ when  $$|\xi| \leq  10 \eo^{-1} \, \sup_{(\zeta,s) \in U_{\eo} \times [-\pi,\pi]} | \rho^{\C}(q^{\C}(\zeta),r(s) )  |.$$

 In this section we prove
\begin{prop} \label{decay estimate} Let  $H \subset \Omega$ be $C^{\omega}$ interior curve with $\dist(H,\partial \Omega) <\delta(\eo)$ and $\partial C_{\eo}$ be a  curve satisfying (\ref{oval}). Then, assuming  $\delta(\epsilon_0) >0$ is sufficiently small and $k \in \Z^+,$  there is a constant $C_k(\eo) >0$ such that for $h \in (0,h(\eo)],$
$$  \| \, [Op_{h}(1-\chi)   \tilde{N}^{\C}](\cdot,\cdot;h) \, \|_{C^{k}([-\pi,\pi] \times [-\pi,\pi])}   = {\mathcal O}(e^{-C_k(\eo)/h}). $$
\end{prop}

\begin{proof} In light of the complexified potential layer formula in (\ref{holcont3}), we  substitute the complex WKB asymptotics for $N^{\C}(q^{\C}(\zeta),r(s),h)$ in (\ref{complex wkb}) and use
the Cauchy integral formula to deform contours of integration.

From  (\ref{holcont3}) and (\ref{complex wkb}), one gets that
\begin{equation} \label{integralsum} \begin{array}{ll}
[Op_h(1-\chi )\tilde{N}^{\C}](t,s,h) \\ \\
 =(2\pi )^{-2} \sum\limits_{\xi \in h\Z}  \int_{-\pi}^{\pi}e^{i [ (t-t') \xi + \rho^{\C}( q^{\C} (\kappa(t')), r(s)) ]/h}  \, (1-\chi)(\xi) \, a^{\C}(  \kappa(t'),r(s);h)  \, dt'.  \end{array} \end{equation}
Consider the complex phase
$$ \Psi^{\C}(t,t',s):= (t-t') \xi + \rho^{\C}( q^{\C} (\kappa(t')), r(s)).$$

For
simplicity, write $\rho^{\C}(t',s) $ for
$\rho^{\C}(q^{\C}(\kappa(t')),r(s))$. Consider for $\xi \in h\Z$  the deformed contour
\begin{equation} \label{new contour}
 \omega_{\xi}(t') = t' - i \frac{\eo}{2} \sgn(\xi) .\end{equation}
 where  $ (t,t',s) \in [-\pi,\pi]^3.$
The deformed phase function
\begin{equation} \label{Taylor}
  \Psi(t,\omega_{\xi}(t'),s) = \Psi \left( t, t' - i \frac{\eo}{2} \sgn(\xi) ,s \right) = (t-t')\xi + i \frac{\eo}{2} |\xi| + \rho^{\C}(\omega_{\xi}(t'),s).   \end{equation}
Since $|\xi| \geq  10 \eo^{-1} \, \sup\limits_{(\zeta,s) \in U_{\eo} \times
[-\pi,\pi]} | \rho^{\C}(q^{\C}(\zeta),r(s))  |$ when $\xi \in
\text{supp}(1-\chi),$  it follows from (\ref{Taylor}) that
\begin{equation} \label{imag bound}
\Im \Psi(t,\omega_{\xi}(t'),s)  \geq  4  \sup\limits_{(\zeta,s) \in U_{\eo}
\times [-\pi,\pi]}| \rho^{\C}(q^{\C}(\zeta),r(s))  |
\gtrapprox \eo \end{equation} uniformly for $(t,t',s) \in
[-\pi,\pi]^{3}.$ Moreover, for $|\xi| \gg 1$ it also follows from
(\ref{Taylor}) that
\begin{equation} \label{imag bound2}
\Im \Psi(t,\omega_{\xi}(t'),s) = \frac{\eo}{2} |\xi| + {\mathcal
O}(1) \geq  \frac{\eo}{3} |\xi|. \end{equation}
Using Cauchy's theorem, we deform the $t'$-contour of
integration in (\ref{integralsum}) to get
\begin{equation} \label{integralsum2} \begin{array}{ll}
[Op_h(1-\chi )\tilde{N}^{\C}](t,s,h) \\ \\
 =(2\pi)^{-2} \sum_{\xi \in h\Z}  \int_{-\pi}^{\pi}  e^{i \Psi (t,\omega_{\xi}(t'),s;\xi)/h } \, (1-\chi)(\xi) \, a^{\C}( \kappa^{\C}(\omega_{\xi}(t')) , r(s) ; h )  \, dt'  \end{array} \end{equation}
where the imaginary part of the deformed phase function
$\Psi(t,\omega_{\xi}(t'),s)$ satisfies (\ref{imag bound}). It
follows from (\ref{imag bound}) and (\ref{imag bound2}) that for appropriate $C(\eo) \gtrapprox \eo,$
\begin{equation} \label{integralsum3}
| [Op_h(1-\chi )\tilde{N}^{\C}(t,s,h)  | \leq e^{- \frac{
C(\eo)}{h} }  \times \left( \sum_{|\xi| \geq1} e^{ - \frac{\eo}{4
h} |\xi|} \right) = {\mathcal O}( e^{- \frac{C(\eo)}{h}
}).\end{equation}

The argument for the higher $C^k$-norms is basically the same since the complex phase function $\Psi^{\C}(t,t',s)$ is unchanged. The derivatives $\partial_s^{\alpha}$ and $\partial_t^{\beta}$ just create additional polynomial powers in $h^{-1}$ in the amplitude $a^{\C}(\cdot,\cdot;h).$

\end{proof}

\begin{rem} \label{R}  For future reference (see proof of Theorem \ref{mainthm1} below),  we note that when $\chi_R \in C^{\infty}_{0}(\R)$ with $\chi_R(\xi) =1$ for $|\xi| < R$ and  supp $\chi_R \subset \{ \xi; |\xi| < 2R \},$ it is clear from  (\ref{imag bound2}) that
\begin{equation} \label{integralsum3}
\| Op_h(1-\chi_R )\tilde{N}^{\C}(\cdot,\cdot,h) \|_{C^k} = {\mathcal O}_k( e^{- \frac{C_R(\eo)}{h}
}),\end{equation}
where $C_R(\eo) \gtrapprox R$ as $R \rightarrow \infty.$
\end{rem}


\section{Proof of  Theorem 1.1} \label{Mainthm1}
 \begin{proof} 
  Let $\chi_{R} \in C_{0}^{\infty}(\R;[0,1])$ be a frequency cutoff as in Proposition \ref{decay estimate} with $\chi_{R}(\xi) = 1$ for $|\xi| \leq R$ and $\chi_{R}(\xi) = 0$ for $|\xi| \geq 2R.$ To simplify notation, in the following we  continue to write $L^2_{\eo} = L^{2} (\partial C_{\eo})$ (resp.  $L^2 = L^2 ([-\pi,\pi])$) and the corresponding unit speed parameterizations  are $t \mapsto \kappa(t)$ (resp. $t \mapsto q(t)).$\\

 We recall that the basic frequency function estimate  gives
 \begin{align*}
 n(h,H) & \leq h^{-1}  \frac{ \| h \partial_{T}  u_{h}^{H,\C} \|_{ L^2_{\eo} }  }{  \| u_{h}^{H,\C} \|_{ L^2_{\eo} } } \\
& \leq h^{-1} \Big( \frac{ \|  Op_h(\chi_R) (h \partial_{T})   u_{h}^{H,\C} \|_{ L^2_{\eo} }  }{  \| u_{h}^{H,\C} \|_{ L^2_{\eo} } }  \, + \,   \frac{ \| (1-Op_h(\chi_R)) (h \partial_{T})  u_{h}^{H,\C}\|_{ L^2_{\eo} }  }{  \| u_{h,H}^{\C} \|_{ L^2_{\eo} } }  \, \Big).
\end{align*}
From Proposition \ref{decay estimate} and Cauchy-Schwarz, it
follows that
\begin{equation} \label{residual1}
 \frac{ \| (1-Op_h(\chi_R)) h\partial_{T}  u_{h}^{H,\C}\|_{ L^2_{\eo} }  }{  \| u_{h}^{H,\C} \|_{ L^2_{\eo} } }  = {\mathcal O} \left( \frac{ e^{ - \frac{ C_R(\eo)}{h}  } }{  \| u_{h}^{H,\C} \|_{ L^2_{\eo} } }  \right) . \end{equation}
In the last line of (\ref{residual1}) we have used that $\| u_{h}^{\partial \Omega}\|_{L^2} = {\mathcal O}(h^{-\alpha})$ for $\alpha >0$ (for example, Tataru's sharp bound \cite{Ta} gives $\alpha = 1/3$).
Since $u_{h}^{H,\C}(t)$ is holomorphic for all $t \in
S_{2\eo,2\pi},$ it follows from the Cauchy integral formula (see figure 1) and
the Cauchy-Schwarz inequality that
\begin{align} \label{residual2}
\sup_{t \in S_{\eo,\pi}} | u_{h}^{H,\C}(t)| & \leq C_2   \cdot\frac{1}{4\pi^2} \left(  \int_{-\pi}^{\pi} \int_{-\pi }^{\pi}  | \kappa(s) - t|^{-2}  \, ds dt  \right)^{\frac{1}{2}}   \cdot \| u_{h}^{H,\C} \|_{L^{2}_{\eo}}  \\ &= {\mathcal O}(1)  \| u_{h}^{H,\C} \|_{L^{2}_{\eo}}\notag .
 \end{align}
In (\ref{residual2}) we use that $\partial C_{\eo}$ and $\overline{S_{\eo,\pi}}$ are disjoint so that $f(s,t) = |\kappa(s) - t|^{-1} \in L^2([-\pi,\pi] \times [-\pi,\pi]).$
Substitution of (\ref{residual2}) in (\ref{residual1}) then
implies that
\begin{equation} \label{residual 3}
 \frac{ \| (1-Op_h(\chi_R)) h\partial_{T}  u_{h}^{H,\C}\|_{ L^2_{\eo} }  }{  \| u_{h}^{H,\C} \|_{ L^2_{\eo} } }  = {\mathcal O}\left(  e^{ \frac{ - C_R(\eo)}{h} } \| u_{h}^{H,\C} \|^{-1}_{L^{\infty}(S_{\eo,\pi})} \right) = {\mathcal O} ( e^{ \frac{ - C_R(\eo) + C_0}{h} }), \end{equation}
 since by assumption $\| u_{h}^{H,\C} \|_{L^{\infty}(S_{\eo,\pi})}  \geq e^{-\frac{C_0}{h}}$ for some $C_0>0.$
 Since $ Op_h(\chi_R) (h \partial_T) \in Op_h(S^{0,-\infty}(T^*H)),$ it follows by $L^2$-boundedness that
  \begin{equation} \label{main}
 \frac{ \| Op_h(\chi_R) h \partial_{T}  u_{h}^{H,\C}\|_{ L^2_{\eo} }  }{  \| u_{h}^{H,\C} \|_{ L^2_{\eo} } }  = {\mathcal O}_{R,\eo}(1). \end{equation}
The constant $C_{R}(\eo) \gtrapprox R$ as $R \rightarrow \infty,$  and so, the proof of Theorem \ref{mainthm1} follows from
(\ref{residual 3}) and (\ref{main}),
by choosing $R$ sufficiently large so that $C_{R}(\eo) - C_0 >0$ in (\ref{residual 3}). \end{proof}
 \begin{figure}[h]
\centering
\includegraphics[width=0.8\textwidth]{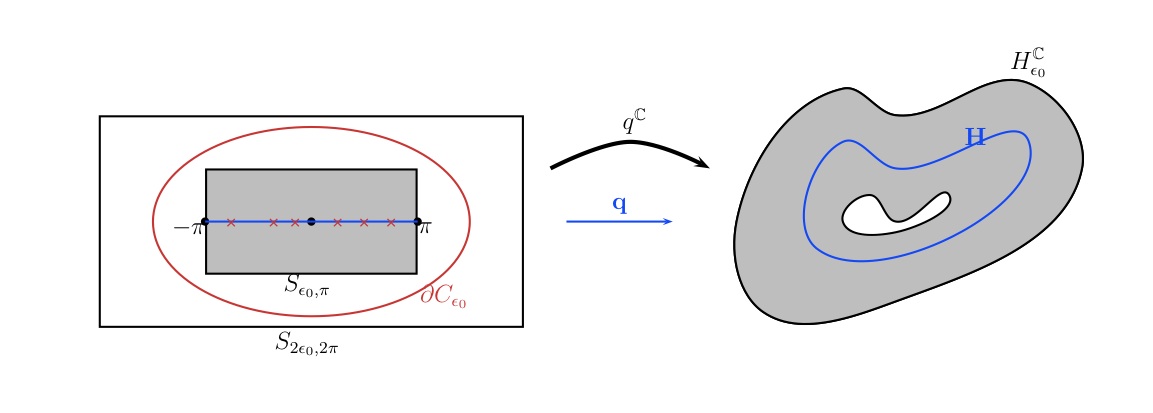}
\caption{}
\end{figure}

\section{Proofs of Theorems \ref{mainthm2} and \ref{mainthm3} }
\begin{proof}  The key ingredient in the proofs of  Theorems \ref{mainthm2} and \ref{mainthm3} is the following operator bound:


\begin{prop}\label{mainthm2prop} Let $H \subset \Omega$ be a closed, strictly convex, interior  real analytic curve. Let $N^{\C}(q^{\C}, r;h)$ be the holomorphic
extension of $N(q, r;h )$ in the $q$ variables to $H_{\eo}^{\C}$ with the corresponding operator
$$N^{\C}(h): L^{2}(\partial \Omega; ds) \rightarrow L^{2}( H_{\eo}^{\C}; \, e^{- \, \frac{S(t)}{h}  }  \, dt d\overline{t}),$$
where $H_{\eo}^{\C} = \{ q^{\C}(t); |\Im t| \leq \eo \}.$   Let  $a \in C^{\infty}_{0}(H_{\eo}^{\C})$ with
\begin{align} \label{supportprop}  \text{supp}  \, a  \subset \{ q^{\C}(t) \in  H_{\eo}^{\C}; \, \frac{\eo}{6} \leq \Im t \leq \frac{5\eo}{6} \}. \end{align}
 Then for $h \in (0, h_0(\eo)]$, and $\eo >0$ sufficiently small, there exists an {\em associated symbol} $a_{\mathcal G} \in C^{\infty}_{0}(B^*\partial \Omega) \subset S^{0,-\infty}(T^* \partial \Omega \times (0,h_0])$ such that
$$
h^{-1/2}  \, N^{\C}(h)^{*} e^{-2S/h} a N^{\C}(h) =  Op_{h}(a_{\mathcal G}) +  R(h).
$$
For $(s,\sigma) \in B^*\partial \Omega,$ the symbol
\begin{equation}\label{aG}
a_{\mathcal G}(s,\sigma) = \frac{1}{\sqrt{2}} a(\Re t(s,\sigma), \Im t(s,\sigma)) \, \kappa_H^{-2}(Y(s)) \,  |\Im t(s,\sigma)|^{-1} \, \gamma^{2}(s,\sigma),
\end{equation}
where, $\gamma(s,\sigma) = \sqrt{1-|\sigma|^2}$ and
\begin{align} \label{correspondence}
Y(s) &= \Re t(s,\sigma)  ( 1 + {\mathcal O}(|\Im t(s,\sigma)|) ), \nonumber \\
 \sigma &=  - \langle \omega(s,Y(s)), T_{\partial \Omega}(s) \rangle  + \frac{\kappa^2_H(Y(s))}{2} d_s Y(s) |\Im t(s,\sigma)|^2 ( 1+ {\mathcal O}(|\Im t(s,\sigma)|) ). \end{align}
Moreover, the remainder satisfies
 $$\| R(h) \|_{L^{2}(\partial \Omega)  \rightarrow L^{2}(\partial \Omega) } = {\mathcal O}(h).$$
  \end{prop}

\begin{rem} We note that the support properties of $a(t)$ in Proposition \ref{mainthm2prop} are stated for concreteness and can be replaced with any amplitude supported in a strip not containing a real interval. In particular, for $a(t)= \chi_{\eo}(\Im t)$ where $\chi_{\eo} \in C^{\infty}_{0} $ is supported in any strip $ \{\eo - \delta < \Im t < \eo \}$ with $0<\delta < \eo$ arbitrarily small, the operator $$P(h) =  [ h^{-1/4} e^{-S/h} \chi_{\eo} N^{\C}(h)]^* \cdot [ h^{-1/4} e^{-S/h} \chi_{\eo} N^{\C}(h)]$$ in (\ref{poperator})  satisfies $$P(h) = Op_{h}(a_{\mathcal G}) +  R(h),$$ where $a_{\mathcal G}$ is as in (\ref{aG}) with $a(t)= \chi_{\eo}(\Im t)$. \end{rem}

\begin{rem} Since $ | \langle \omega(s,Y(s)), T_{\partial \Omega}(s) \rangle | < 1,$ it follows from (\ref{correspondence}) and the support assumptions on $a(\Re t, \Im t)$ in (\ref{supportprop})  that for $\eo >0$ small, $a_{\gcal} \in C^{\infty}_{0}(B^*\partial \Omega)$ (ie. has support disjoint from the tangential set $S^*\partial \Omega$).\\
\end{rem}

The proof of proposition \ref{mainthm2prop} is rather technical and to avoid breaking the exposition at this point we defer the proof to section 8.
As an immediate consequence proposition \ref{mainthm2prop} we have the following corollary :

\begin{cor}\label{propcor}
Assume $\Omega$ is a smooth convex bounded domain and  the interior curve $H$ is strictly convex. Let $\chi_{\eo} \in C^{\infty}_{0}(S_{\eo,\pi})$ supported in the strip  $ \{\eo - \delta < |\Im t| < \eo \}$ with $0<\delta < \eo.$ Then, for $\eo >0$ sufficiently small, there exists an order zero pseudodifferential  operator $P(h)$ such that $$h^{-1/2} \int\int_{S_{\eo,\pi} }    e^{-2 \frac{ S(t)}{h} } | u_{h}^{H,\C}(t) |^{2} \, \chi_{\eo}(t)\, dt d\overline{t}  = \langle P(h) \phi_{h}^{\partial \Omega}, \phi_h^{\partial \Omega} \rangle_{L^2},$$
such that the principal symbol $\sigma(P(h))$ satisfies
$$ \int_{B^*\partial \Omega} \sigma(P(h))\gamma^{-1} \, dy d\eta \geq C_0(\Omega,H,\eo) >0, $$ where $\gamma(y,\eta) = \sqrt{1-|\eta|^2}.$
\end{cor}

\begin{proof}
Given $P(h) = h^{-1/2}[ e^{-S/h} \chi_{\eo} N^{\C}(h)]^* \cdot [ e^{-S/h} \chi_{\eo} N^{\C}(h)],$  the result follows by an application of the complexified potential layer formula (\ref{holcont3}) and proposition \ref{mainthm2prop}. 
\end{proof} 

 Assuming Proposition \ref{mainthm2prop} for the moment, as discussed in subsection \ref{mainthm2proof}, we claim the proofs of Theorems \ref{mainthm2}  and \ref{mainthm3} follow easily from Theorem \ref{mainthm1} and Corollary \ref{propcor}.  
 \subsubsection{Proof of Theorem \ref{mainthm3}}  \label{eigenfnestimates}
 The lower bound in Theorem \ref{mainthm3} follows from Corollary \ref{propcor} by taking supremum inside the integral.  First, it is clear from the proof of Proposition \ref{mainthm2proof} that  the interval $[\eo/6, 5\eo/6]$ can be replaced by  any interval of the form $I(\eo,\delta):=[\eo - \delta, \eo - \delta/2] $  with $0< \delta < \eo.$ Without loss of generality, we can further assume that  $\chi_{\eo} \in C^{\infty}_{0}( [\eo - 3 \delta, \eo - \delta/3];[0,1])$ with $\chi_{\eo}(\Im t) = 1$ for $\Im t \in I(\eo,\delta).$ Thus,
 \begin{align} \label{thm2proof1}
 h^{-1/2}\max_{ q^{\C}(t) \in H_{\eo}^{\C}} |u_{h}^{H,\C}(t)|^2 \times \Big( \int_{-\pi}^{\pi} \int_{\eo-3\delta}^{\eo - \delta/3} e^{- 2 S(t)/h} dt d\overline{t} \Big)& \geq   h^{-1/2}\int_{S_{\eo,\pi}} e^{-2S(t)/h} \chi_{\eo}(t) |u_{h}^{H,\C}(t)|^2 \, dt d\overline{t} \nonumber \\
 &\sim_{h \to 0^+} \langle P(h) u_{h}^{\partial \Omega}, u_{h}^{\partial \Omega} \rangle_{L^2(\partial \Omega)} \nonumber \\
 &\sim_{h \to 0^+} c_{H,\eo}. \end{align}
 In the last line we have used the QER property of the boundary traces $u_h^{\partial \Omega}$ of the QE sequence of interior eigenfunctions to obtain  $$ \langle P(h) \phi_h^{\partial \Omega}, \phi_h^{\partial \Omega} \rangle \sim_{h \rightarrow 0^+} \int_{B^*\partial \Omega} \sigma(P(h)) \gamma^{-1} \, dy d\eta \geq c_{H,\eo}  >0. $$ Since for $\Im t \in I(\eo,\delta),  S(t) = \Im t + O( (\Im t)^3)$, $d_{\Im t} S(t)  = 1 + O(|\Im t|^2),$  by making the change of variables $(\Re t, \Im t) \to (\Re t, S(t)),$ it follows that
 $$ \int_{-\pi}^{\pi} \int_{\eo-3\delta}^{\eo - \delta/3} e^{-2S(t)/h} dt d\overline{t}   \leq C_{\eo,\delta} h e^{-2S(\eo-3\delta)/h}. $$  Thus, it follows from (\ref{thm2proof1}) that 
 \begin{equation} \label{thm2proof2}
  \max_{ q^{\C}(t) \in H_{\eo}^{\C}} |u_{h}^{H,\C}(t)|  \geq C'_{H,\eo,\delta} h^{-1/4} e^{[ S(\eo-3\delta)]/h}. \end{equation}
  Since one can choose $\delta \in (0, \eo)$ arbitrarily,  the  lower bound in Theorem \ref{mainthm3} follows from (\ref{thm2proof2}) and the polynomial factor $h^{-1/4}$ is irrelevant since it gets absorbed into the exponential.

  As for the upper bound, we simply use the complexified potential layer formula (\ref{holcont3}) and apply Cauchy-Schwarz to get
  \begin{align} \label{thm2proof3}
  |u_{h}^{H,\C}(t)|& \leq (2\pi h)^{-1/2}  \Big| \int_{\partial \Omega} e^{i/h \rho^{\C}(t,s)} a(t,s;h) u_{h}^{\partial \Omega}(s) d\sigma(s) \Big| \nonumber \\
  &\leq C_{H} h^{-1/2} e^{ S(\eo)/h} \| u_{h}^{\partial \Omega} \|_{L^2} \leq C_{H,\eo} h^{-1/2} e^{S(\eo)/h}.& \end{align}
 In the last step, we used the a priori bound $\| u_{h}^{\partial \Omega} \|_{L^2} = O(1)$  combined the fact that $ \max_{( q^{\C},r) \in H_{\eo}^{\C} \times \partial \Omega } e^{- \Im \rho^{\C}(q^{\C},r)/h }  \leq e^{S(\eo)/h}.$ 
The upper bound  for  $\| u_h^{\partial \Omega} \|_{L^2}$ follows from the fact that the boundary restrictions $u_h^{\partial \Omega} = \phi_h |_{\partial \Omega}$ are themselves QE in the sense of (\ref{qer}).   In the Dirichlet case, the Rellich formula gives $\| h \partial_{\nu} \phi_h \|_{L^2(\partial \Omega)} = O(1)$ and so,  the upper bound in (\ref{thm2proof3}) is also $O( h^{-1/2} e^{S(\eo)/h}).$ \qed

 \subsubsection{Proof of Theorem \ref{mainthm2}} From the lower bound in Theorem \ref{mainthm3}, for a real-analytic positively curved $H,$ sufficiently small $\eo>0$  and any $\delta >0,$ it follows that
 \begin{equation} \label{lower}
  \max_{ q^{\C}(t) \in H_{\eo}^{\C}} |u_{h}^{H,\C}(t)|^2  \geq C_{H,\eo,\delta} e^{[2 S(\eo - \delta)]/h},\end{equation}
   it is obvious that any such curve is good in the sense of (\ref{revgood}) and consequently, Theorem \ref{mainthm2} follows from Theorem \ref{mainthm1}. \end{proof}
 We note that the lower bound (\ref{lower}) is much stronger than what is required for Theorem \ref{mainthm2} since it shows that the tube maxima of holomorphic continuations of eigenfunction restrictions actually {\em grow} exponentially in the tube radius.
 
  Also, in regards to (\ref{thm2proof1}), as we have already indicated in subsection \ref{mainthm2proof}, it follows from the Rellich commutator argument in \cite{Bu}  that quantum ergodicity  of the interior eigenfunctions $\phi_h$ imply that the boundary restrictions $\phi_{h}^{\partial \Omega}$  have the analogous quantum ergodic restriction  property  in (\ref{qer}). We note that the last statement is not necessarily true if one replaces $\partial \Omega$ by an arbitrary interior curve, $H.$

Before proving Propositioin \ref{mainthm2prop}, we will need some background on asymptotics of the complexified potential layer operator $N^{\C}(h),$ its relation to the glancing map $Y: \partial \Omega \to H$ and complexification.

\section{Asymptotics for the complexified potential layer operator $N^{\C}(h)$} \label{guts}

To simplify the writing somewhat, we assume throughout this section that $\partial \Omega$ is smooth. The case of boundaries with corners is discussed in the final subsection \ref{corners}.

Abusing notation somewhat we let 
\begin{equation} \label{complexified phase2}
\rho^{\C}(t,s) := \rho^{\C}( q^{\C}(t),r(s))  \end{equation} for $(t,s) \in S_{2\eo,2\pi} \times [-\pi,\pi]$ where the RHS in (\ref{complexified phase2}) is complexified distance function (see (\ref{complexified phase})). We define  the weight function
\begin{equation} \label{Sweight}
S(t):= \max_{s \in [-\pi,\pi]} \Re \,[ i \rho^{\C}(t,s) ],
\end{equation}
one has the following

\begin{lem} \label{complexWKB}
For $q^{\C}(t) \in H^{\C}_{\eo}$ and with  the weight function $S(t)$ in (\ref{Sweight}), there exist $b_j^{\C} (\cdot,s) \in O(S_{2\eo,2\pi} ; C^{\omega}(\R/2\pi \Z) ) ; j \geq 0$ such that
\begin{equation}\label{CWKB}
e^{-S(t)/h}  \cdot  N^{\C}(q^{\C}(t),r(s);h) = (2\pi h)^{-1/2} \exp \, \left(  [ i \rho^{\C}(t,s) - S(t)] /h  \right) \,  \left( \sum_{j=0}^{N} b_{j}^{\C}(t,s) h^{j}  \right) + {\mathcal O}(h^{N+1}).
\end{equation}
\end{lem}

\begin{proof} The lemma is an immediate consequence of Proposition \ref{N} and (\ref{complex wkb}) since
$$ - S(t) + \Re ( i \rho^{\C}(t,s)) \leq 0, \,\, (t,s) \in S_{2\eo,2\pi} \times [-\pi,\pi].$$ \end{proof}

The main step in the proof of Proposition \ref{mainthm2prop} is an analysis of  the asymptotics of the composite operators $P(h):C^{\infty}(\partial \Omega) \rightarrow C^{\infty}(\partial \Omega),$ where
$$ P(h) = h^{-1/2} [ e^{-S/h} \chi_{\eo} N^{\C}(h)]^* \cdot [ e^{-S/h} \chi_{\eo} N^{\C}(h)]. $$
For this, one needs a detailed analysis of the complex phase function on the right hand side of (\ref{CWKB}). We begin with

\subsection{Asymptotic expansion of $\rho^{\C}(t,s)$.}

 Let $T_{H}(s) = d_s q(s)$ be the unit tangent to $H$ and  $\nu_{H}(s)$ the unit outward normal to $H$. Throughout the paper, $\kappa_H(s)$  denotes the scalar curvature of $H$. In the following, it will be useful
  to define the relative displacement vector
  $$ \omega(s, \Re t)  :=  \frac{q(\Re t) - r(s)}{|q(\Re t) -r(s)|}.$$
   From the Frenet-Serret formulas, we get that for $\eo >0$ small,  the holomorphic continuation $q^{\C}$ of the parametrization $q$ of $H$ satisfies for $|\Im t | \leq \eo,$
\begin{equation} \label{taylor1} \begin{array}{ll}
  q^{\C}(\Re t + i \Im t) - r(s) = q(\Re t) - r(s) +  i \Im t  \,  T_H(\Re t) - \frac{1}{2} \kappa_H(\Re t) | \Im t| ^{2} \nu_H(\Re t) \\ \\
 - \frac{i}{6} (\Im t)^{3}  [ \kappa_H'(\Re t) \nu_H(\Re t) - \kappa_H^{2}(\Re t) T_H(\Re t)] \, + {\mathcal O}( |\Im t |^{4}). \end{array} \end{equation}
Similarly, when $|t-s| \leq  \eo,$ one also has the expansion
\begin{equation} \label{Taylor2} \begin{array}{ll}
  q^{\C}(\Re t + i \Im t) -  q(s) = (\Re t + i \Im t -s) T_H(s) + \frac{1}{2} \kappa_H(s) ( \Re t  + i \Im t -s)^{2} \nu_H(s) \\ \\
 + \frac{1}{6} [ \kappa_H'(s) \nu_H(s) - \kappa_H^{2}(s) T_H(s)] \, ( \Re t + i\Im t -s)^{3} + {\mathcal O}( |\Re t + i\Im t -s|^{4}). \end{array} \end{equation}
Both (\ref{taylor1}) and (\ref{Taylor2}) will be useful at different points in our analysis;  the former when determining growth of functions in $\Im t$ and the latter when estimating joint growth in $\Re t -s$ and $\Im t$.

Let  $\langle, \rangle: \C \times \C \rightarrow \C$ be the standard complex bilinear extension of the Cartesian inner product on $\R \times \R.$ A direct computation using (\ref{taylor1}) gives
\begin{align} \label{imag} 
 \nonumber&\Im \rho^{\C}(t,s) \\ \nonumber &= \langle \omega(s,\Re t), T_{H}(\Re t) \rangle  \, (\Im t)  - \Big( \frac{1}{6} \langle \kappa_{H}'(\Re t) \nu_{H}(\Re t) - \kappa_H^{2}(\Re t) T_{H}(\Re t), \omega(s,\Re t) \rangle \\ \nonumber
&\,\,\,\,\,\,\,- \frac{1}{2} \kappa_{H}(\Re t) \langle \nu_{H}(\Re t), \omega(s,\Re t) \rangle
+   \frac{1}{2} \langle \omega(s,\Re t), T_{H}(\Re t) \rangle  \,  |q(\Re t) - r(s)|^{-2} \\ 
&\,\,\,\,\,- \frac{1}{2} |q(\Re t) - r(s)|^{-2}  \langle \omega(s,\Re t), T_{H}(\Re t) \rangle^{3}  \, \Big)\, (\Im t)^{3} + {\mathcal O}(|\Im t|^{5}).
 \end{align}
It  follows that at a critical point $ s =s^*(t)$ of $\Im \rho^{\C}(t,s)$
\begin{equation} \label{term.9}
 \partial_{s} \Im \rho^{\C}(t,s^*(t)) = 0, \end{equation}
and when $\Im t \neq 0,$ we have
\begin{equation} \label{term1}
\langle \partial_s \omega(s^*(t),\Re t), T_H(\Re t) \rangle + {\mathcal O}(|\Im t|^{2}) = 0.
\end{equation}
Moreover, when equation (\ref{term1}) is satisfied, we have

\begin{lem} \label{weight1}
Let  $t \in [-\pi,\pi] + i [ \frac{\eo}{2}, \eo]$  solve the critical point equation in (\ref{term1}). Then, for $\eo >0$ sufficiently small,
$$ | \langle T_{H}(\Re t), \omega(s^*(t),\Re t) \rangle | =1 + {\mathcal O}(|\Im t|^{2}).$$
\end{lem}
\begin{proof}
Carrying out the $s$-differentiation gives
\begin{align} \label{diff formula} 
\nonumber & \langle \partial_sd (s,\Re t), T_{H}(\Re t) \rangle  \\
&=|q(\Re t)-r(s)|^{-1}  \Big(  \langle T_{\partial \Omega}(s), T_{H}(\Re t) \rangle  
 - \langle T_{\partial \Omega}(s), \omega(s, \Re t) \rangle \cdot \langle T_{H}(\Re t), \omega(s,\Re t) \rangle \, \Big), \end{align}
 where $T_{\partial \Omega}(s) = d_{s} r(s)$ the unit tangent to $\partial \Omega$.\\
Since $|T_{H}(\Re t)| = |T_{\partial \Omega}(s)| = |\omega(s,\Re t)| =1$, it follows from (\ref{diff formula}) and the cosine law $\cos (\theta_1 + \theta_2) = \cos \theta_1 \cos \theta_2 - \sin \theta_1 \sin \theta_2$ that $\partial_s \langle d, T_H \rangle = 0$ if and only if either
$$(i) \,\,\,| \langle T_{\partial \Omega}(s), \omega(s,\Re t) \rangle | = 1$$
or $$ (ii)\,\,\,| \langle T_{H}(\Re t), \omega(s,\Re t) \rangle | =1.$$ The identity (i) is never satisfied since $H$ is by assumption an interior   curve and $\partial \Omega$ is convex, so it is supported by the tangent line at each point of the boundary.  As a result, (ii) must hold and this finishes the proof.\end{proof}
Given, its geometric significance, in view of Lemma \ref{weight1} it makes sense to single out the points $s=s(\Re t)$ which solve the approximate critical point equation
\begin{equation} \label{approxcrit}
\langle T_{H}(\Re t), \omega(s(\Re t),\Re t) \rangle  = - 1. \end{equation}
Geometrically, $q(s(\Re t)) \in \partial \Omega$ is  the boundary intersection of the billiard trajectory in $\Omega$ that tangentially glances $H \subset \Omega$ at $q(\Re t)$. By convexity there are two such points on the boundary and the condition  $\langle \omega(s,\Re t), T_{H}(\Re t) \rangle = -1$  uniquely specifies the point.

\begin{rem} In the next section, we improve the result in Lemma \ref{weight1} and show that in fact
$| \langle T_H(\Re t), \omega(s(\Re t), \Re t) \rangle | = 1 + {\mathcal O}(|\Im t|^4)$ which implies that the holomorphic continuation $s(t)$ of the  geometric solution of (\ref{approxcrit}) agrees to ${\mathcal O}(|\Im t|^5)$-error with the exact critical point $s^*(t)$ in (\ref{term1}). We then use this fact to determine the asymptotics of the weight $S(t)$ to ${\mathcal O}(|\Im t|^5)$-accuracy.
\end{rem}

\subsection{Glancing sets relative to $H$}

We start by defining the  glancing set (and associated glancing map) relative to $H$.  The real part of  the complex phase $ \Re i \rho^{\C}(t,s)$ attains an approximate maximum at $s= Y^{-1}(t)$ where, $Y^{-1}$  denotes the inverse glancing map (see (\ref{grazing}) below). As we show  in (\ref{exact}) below, modulo ${\mathcal O}(|\Im t|^5)$-error terms,  the weight function $S(t)$  equals  $\Re i \rho^{\C}(t,Y^{-1}(t))$. The points $Y^{-1}(t)$ have a simple geometric characterization in terms of  glancing sets relative to $H$, which we now describe. Unless specified otherwise, when $t$ is complex we assume in the following that $\Im t \geq 0$.

\begin{lem} \label{fact1}
For fixed $s \in [-\pi,\pi ]$ let $Y(s)$ be a solution of $T_{H}(\cdot ) = - \omega(s, \cdot )$. Then, the map $Y: [-\pi,\pi ] \rightarrow [-\pi,\pi]$ defined by $s \mapsto Y(s)$
induces a real-analytic diffeomorphism of $H$ with $\partial \Omega.$ By an abuse of notation, we also denote the latter map by $Y.$ 
\end{lem}
\begin{proof} The equation $| \langle T_{H}(\Re t), \omega(s, \Re t) \rangle | = 1$ is equivalent to
\begin{equation} \label{criteqn}
\langle \nu_{H}(\Re t), q(\Re t) - r(s) \rangle = 0. \end{equation}

Unlike the defining equation in (\ref{approxcrit}),  $(\ref{criteqn})$ has the advantage of being non-degenerate in $\Re t$. Indeed, differentiating the left hand side of $(\ref{criteqn})$ with respect to $\Re t$ yields
$$ \kappa_{H}(\Re t) \langle T_{H}(\Re t), q(\Re t) - r(s) \rangle + \langle \nu_{H}(\Re t) , T_{H}(\Re t) \rangle = \kappa_{H}(\Re t) \langle T_{H}(\Re t), q(\Re t) - r(s) \rangle $$
Evaluating the last expression on the right hand side at $\Re t=Y(s)$, implies that
\begin{equation}\begin{array}{ll} \label{**}
\left| \, \partial_{\Re t}  \langle \nu_{H}(\Re t), q(\Re t) - r(s) \rangle \, \right| |_{\Re t = Y(s)} \\ \\ = \kappa_{H}(Y(s)) \, |q(Y(s)) - r(s)| \geq \min_{p \in H} \kappa_{H}(p) \cdot dist (H, \partial \Omega) >0, \end{array}\end{equation}
given that $\kappa_H >0.$ Similarily,
\begin{equation}\begin{array}{ll} \label{***}
\left| \, \partial_{s}  \langle \nu_{H}(\Re t), q(\Re t) - r(s) \rangle \, \right| |_{\Re t = Y(s)} \\ \\ = \left|  \langle \nu_H(\Re t), T_{\partial \Omega}(s) \rangle \right| \geq C(H, \partial \Omega) >0, \end{array}\end{equation}
since $H$ is interior and $\Omega$ is convex. From $(\ref{**})$, the implicit function theorem gives two analytic solution curves $ \Re t \mapsto s_{\pm}(\Re t)$ solving $\langle T_{H}(\Re t), \omega(s_{\pm}(\Re t), \Re t \rangle = \pm 1.$
In view of (\ref{***}), there are two smooth solution curves $s \mapsto Y_{\pm}(s)$ solving $\langle T_{H}(Y_{\pm}(s)), \omega(s, Y_{\pm}(s) \rangle = \pm 1$.  We  choose here $Y(s) = Y_{-}(s)$. In the case where $\Im t < 0$, one chooses $Y(s) = Y_{+}(s)$.  The mapping $Y: \R/2\pi \Z \rightarrow \R/ 2\pi \Z$ is clearly bijective due to the positive curvature of $H$. \end{proof}

 \begin{defn} \label{grazing} Let $\Omega$ be a smooth, bounded convex planar domain and $H$ a strictly convex $C^{\omega}$ curve with $H \cap \partial \Omega = \emptyset.$ We define the  {\em glancing set relative to $H$}  for the billiard flow in $\Omega $ to be the set
 $$ \Sigma  := \{ ( r(s), q(\Re t)) \in \partial \Omega \times H; \,  T_{H}(\Re t) = -  \, \omega(s,\Re t) \}.$$ 
 The associated glancing map $Y: \partial \Omega \to H$ is defined implicitly as the unique  solution of the equation
 $$\langle T_{H}(Y(s)), \omega(s, Y(s)) \rangle = - 1.$$
 In view of Lemma \ref{fact1} it is a global $C^{\omega}$ diffeomorphism of $\partial \Omega$ with $H.$
 \end{defn}
There are several elementary facts about $\Sigma$ that will be needed later on when estimating the various $h$-microlocal pieces of the $N^{\C}(h)$-operator in the course of proving Proposition \ref{mainthm2prop}. The first observation is that in view of Lemma \ref{fact1},
\begin{equation} \label{sigma2}
\Sigma = \{ (r(s), q(Y(s))) \in \partial \Omega \times H \} \end{equation}
is a $C^{\omega}$-graph over $\partial \Omega$ in the product manifold.
Moreover, one also has the following useful fact:
\begin{lem} \label{mixed hessian}
Assume that $H \subset \Omega$ is an interior curve and that $\partial \Omega$ is smooth and convex. Then,
$$\Sigma   \subset  \{ (r(s), q(\Re t)) \in \partial \Omega \times H; \, \partial_{\Re t} \partial_{s}\rho (s,\Re t) = 0  \}.
$$
\end{lem}
\begin{proof} This follows from the formula
$$\partial_s \partial_{\Re t} \rho(s,\Re t) =   \langle \partial_s \omega(s,\Re t), T_{H}(\Re t) \rangle.$$
\end{proof}

We denote the canonical transformation induced by the  diffeomorphism $s \mapsto Y(s)$
by
\begin{equation} \label{canonical} \begin{array}{ll}
\zeta_{H}: T^* \partial \Omega \rightarrow T^*H, \\ \\  \zeta_H(s,\sigma) = (y,\eta); \,\,
 y = Y(s), \eta = d_s Y(s)^{-1} \sigma. \end{array} \end{equation}





\subsection{Taylor expansion of $\rho^{\C}(t,s)$ around glancing points.}

Before analyzing the composite operator $N^{\C*}(h) a N^{\C}(h),$ we collect here some  asymptotic formulas for the real and imaginary parts of $\rho^{\C}(t,s)$
 which are useful when $|t - Y(s) | \ll 1$.

 \begin{lem} \label{asymptotic}
 Let $(t,Y(s)) \in ([-\pi,\pi] + i[\frac{\eo}{2}, \eo]) \times [-\pi,\pi]$ where $Y:[-\pi,\pi] \rightarrow [-\pi, \pi]$ is the diffeomorphism in  Lemma \ref{fact1}. Then for  $| t - Y(s)| \leq \eo$ and $\eo >0$ sufficiently small,
 \small{$$  \Re \rho^{\C}(t,s) = |q(Y(s))- r(s)| - (\Re t -Y(s)) \, \left( 1 + \frac{1}{2} \kappa_H^2(Y(s)) |\Im t|^{2} \right) +  \max_{\alpha + \beta = 4}  {\mathcal O}(|\Re t -Y(s)|^{\alpha} |\Im t|^{\beta}),$$}
$$ \Im \rho^{\C}(t,s) = - \Im t  + \frac{\kappa_H(Y(s))}{2} (\Re t -Y(s))^{2} \Im t - \frac{1}{6} \kappa_H^{2}(Y(s)) (\Im t)^{3} + \max_{\gamma + \delta = 5} {\mathcal O}(|\Re t -Y(s)|^{\gamma} |\Im t|^{\delta}).$$
\end{lem}
\normalsize
\begin{proof}
The lemma follows from the formula
\begin{equation*}
 \label{taylor} \rho^{\C}(t,s) = |q(Y(s)) -r(s)| - (t-Y(s))  + \frac{1}{6} \kappa_H^{2}(Y(s)) (t-Y(s))^{3} + {\mathcal O}(|t-Y(s)|^{4})
\end{equation*}
This inturn is a consequence of the Taylor expansion for $ q^{\C}(t) - r(s) =  q^{\C}(t) - q(Y(s)) + q(Y(s)) - r(s)$ around $t=Y(s)$ in (\ref{Taylor2}) using in addition the identities $\langle \omega(s,Y(s)), T_{H}(Y(s)) \rangle = -1$ and $\langle \omega(s,Y(s)), \nu_{H}(Y(s)) \rangle = 0.
$ \end{proof}

\subsection{Weight function}

 We compute in this section the asymptotic formula for the weight function, $S(t).$
\begin{lem} \label{weightlemma}
Let $q^{\C}(t) \in H^{\C}(\eo) - H^{\C}(\eo/2)$ with $\eo >0$ sufficiently small. Then, the weight function  $S(t) = \max_{s \in [-\pi,\pi]}  ( - \Im \rho^{\C}(t,s) )$ has the asymptotic expansion 
$$S(t) = \Im t  +  \frac{1}{6} \kappa_{H}^{2}(\Re t) (\Im t)^{3} + {\mathcal O}(|\Im t|^{5}).$$
\end{lem}
\begin{proof}
We first consider the approximate critical point equation
\begin{equation} \label{approx}
\partial_s \langle T_H(t), \omega(s,t) \rangle = |\Im t|^{4}.
\end{equation}
 When $\Im t = 0$, (\ref{approx}) has the solution $s(\Re t): = Y^{-1}(\Re t)$ in the notation of Lemma \ref{fact1}. Under the assumption that $\kappa_H >0,$  $\partial_{s}^{2} \langle T_H(\Re t), \omega(s,\Re t) \rangle |_{s=Y^{-1}(\Re t)} \geq \frac{1}{C}>0$  and so, by the analytic implicit function theorem, for $\eo >0$ small, $Y^{-1}(\Re t)$ locally extends to a unique real analytic function $ t \mapsto Y^{-1}(t), \, t \in [-\pi,\pi] + i [\frac{\eo}{2}, \eo]$ solving (\ref{approx}).
Substitute the identity  $ \langle \omega(Y^{-1}(\Re t),\Re t), T_{H}(\Re t) \rangle = -1 + |\Im t|^{4}$ into the formula  (\ref{imag}) and  also use that $ \langle \nu_{H}(\Re t), \omega(s(\Re t),t) \rangle  = |\Im t|^{2}$ and
$\partial_s \langle d, \nu_H \rangle = - \langle d, T_H \rangle  (1 - \langle T_H, d \rangle^2 )^{-1/2}  \, \partial_s \langle T_H, d \rangle$
both of which follow from the fact that $ \langle \nu_H, d \rangle = \sqrt{ 1- \langle T_H, d\rangle^{2} }$. Since the last two terms in the $(\Im t)^3$-coefficient on the right hand side of (\ref{imag}) cancel, one gets that
\begin{equation} \label{approx2}
\small{\partial_s \Im \rho^{\C}(t,s) |_{s=Y^{-1}(t)} =  \partial_{s}\langle T_{H}(\Re t),  \omega(s,t) \rangle |_{s = Y^{-1}(t)} \,  \left( \Im t + {\mathcal O}(|\Im t|^{3} \right) + {\mathcal O}(|\Im t|^{5}) = {\mathcal O}(|\Im t|^{5}).}
\end{equation}
\normalsize
Finally, we compare (\ref{approx2}) with exact critical point equation
\begin{equation} \label{exact}
 \partial_s \Im \rho^{\C}(t,s) = 0.
\end{equation}
Let $s = s^*(t)$ be the locally unique analytic solution to (\ref{exact}) with $\eo/2 < \Im t < \eo$ and $\eo >0$ small. Then, again by the Taylor expansion in (\ref{imag}) and the implicit function theorem, it follows that
\begin{equation} \label{exact2}
Y^{-1}(t) - s^*(t) = {\mathcal O}(|\Im t|^{4}).
\end{equation}
Upon substitution of the bounds  (\ref{exact2}) back in (\ref{imag}),  it follows that all terms except the one involving $\frac{1}{6} \kappa_H^{2}(\Re t)$ are absorbed into the ${\mathcal O}(|\Im t|^5)$-error and, in particular,
\begin{equation} \label{approx3}
\Im \rho^{\C}(t,s^*(t)) = - \Im t  - \frac{1}{6} \kappa_H^2(\Re t)  (\Im t)^{3} + {\mathcal O}(|\Im t|^{5}).
\end{equation}
 This gives the stated 
asymptotic formula for $S(t)$ to ${\mathcal O}(|\Im t|^{5})$ error. 
  \end{proof}
\begin{rem} One can repeat the same kind of argument to determine the expansion of $S(t)$ in $\Im t$ to arbitrary accuracy, but the terms rapidly become more cumbersome to compute.
\end{rem}



The value of the weight function (ie. the maximizer of $-\Im \rho^{\C}$ ) is approximately attained  when $s = Y^{-1}(\Re t)  \in [-\pi,\pi ]$ (see (\ref{exact2})).  This suggests  that the dominant part  of the  the $N^{\C}(h)$-operator (resp. $N^{\C*}(h) a N^{\C}(h)$ for any $a \in C^{\infty}_{0} (S_{2\eo,2\pi})$ should come from $\Sigma_{\eo}^{\C}$
 (resp. $\Sigma_{\eo}^{\C} \times \Sigma_{\eo}^{\C}),$ where
 \begin{equation} \label{near glancing}
\Sigma_{\eo}^{\C}:= \{ (t,s) \in S_{2\eo,2\pi} \times [-\pi,\pi]; | t-Y(s)| < \eo \}. \end{equation}
We will call $\Sigma_{\eo}^{\C}$ the  $\eo$-{\em complex glancing set relative to} $H$ in the parameter space $\{ (t,s) \in S_{2\eo,2\pi} \times [-\pi,\pi] \}.$

 One of the first steps in the next section will be to show that the contribution to $N^{\C*}(h) a N^{\C}(h)$ coming from the complement $\{t \in S_{2\eo,2\pi} ; |Y(s)-t| \geq \eo \}$ is of lower order in $h$ in $L^{2}$-norm then the contribution coming from the complex glancing set $\Sigma^{\C}_{\epsilon}.$ This fact relies on some estimates for $\Im \rho^{\C}$  which we collect here.
 Assume that $|\Im t| \lessapprox \eo $ and that  $| \Re t  -  Y(s) | \geq \eo$. Then from (\ref{taylor1}),  it follows that 
\begin{align} \label{taylor2} 
\nonumber&\langle q^{\C}(\Re t + i\Im t) - r(s), q^{\C}(\Re t + i\Im t) - r(s) \rangle \\  \nonumber&= |q(\Re t) - q(Y(s))|^{2} + 2 i \Im t \, \langle q(\Re t) - r(s), T_H(\Re t) \rangle + {\mathcal O}(|\Im t|^{2}) \\
 &\,\,\,\,+ \langle q(\Re t ) - q(Y(s)), q(Y(s)) - r(s) \rangle +  |q(Y(s))-r(s)|^{2}.
 \end{align}
So, taking square roots in (\ref{taylor2}) gives
\begin{align}  \label{taylor3}
\left|\Im \rho^{\C}(\Re t + i \Im t, Y(s)) \right| &= \frac{  | \Im t  \, \langle q(\Re t) - r(s), T_H(\Re t) \rangle + {\mathcal O}(|\Im t|^{2}) | }{|q(\Re t) - r(s)| }\nonumber \\
&= |\Im t|  \, |\langle \omega(Y(s),\Re t), T_H(\Re t) \rangle |  + {\mathcal O}(\Im t^{2}) \nonumber \\ 
&\leq  \frac{1}{C(\eo)} |\Im t|  + {\mathcal O}(|\Im t|^{2}), 
\end{align}
with $C(\eo) >1$. The last estimate in (\ref{taylor3}) follows since $|\Re t- Y(s)| \gtrapprox \eo $ implies that $|\langle \omega(Y(s),\Re t), T_H(\Re t) \rangle| \leq \frac{1}{C(\eo)}$ with $C(\eo) >1$.

\subsection{Taylor expansions of phase functions} 
 We will need to analyze the asymptotics of the operator kernel $P(h)= h^{-1/2}e^{-2S/h} N^{\C}(h)^* a N^{\C}(h)(t,s)$ in various asymptotic regimes; in particular, regions where $|Y(s) - \Re t| \lessapprox |\Im t|$ and the complement $|Y(s)- \Re t| \gtrapprox |\Im t|$. Since $t=Y(s)$ is an (approximate) critical point for  $\Re i \rho^{\C}(t,s),$ the first regions should dominate the asymptotics and the latter should be residual as $h \to 0^+.$  We prove this in detail in the next section where the following Taylor expansions will be used to compute the  asymptotics of the phase function of  $h^{-1/2} e^{-2S/h} N^{\C}(h)^* a N^{\C}(h)(t,s)$ in these regimes. For the convenience of the reader, we collect here the relevant Taylor expansions for phase functions derived above that will be needed in the next section.
\subsubsection{Near-diagonal  expansions}
\begin{equation}  \begin{array}{ll}
  q^{\C}(\Re t + i \Im t) -  q(s) = (\Re t + i \Im t -s) T_H(s) + \frac{1}{2} \kappa_H(s) ( \Re t  + i \Im t -s)^{2} \nu_H(s) \\ \\
 + \frac{1}{6} [ \kappa_H'(s) \nu_H(s) - \kappa_H^{2}(s) T_H(s)] \, ( \Re t + i\Im t -s)^{3} + {\mathcal O}( |\Re t + i\Im t -s|^{4}). \end{array} \end{equation}
\subsubsection{Expansions in the complex tube $H_{\eo}^{\C}$}
\begin{equation} \begin{array}{ll} 
  q^{\C}(\Re t + i \Im t) - r(s) = q(\Re t) - r(s) +  i \Im t  \,  T_H(\Re t) - \frac{1}{2} \kappa_H(\Re t) | \Im t| ^{2} \nu_H(\Re t) \\ \\
 - \frac{i}{6} (\Im t)^{3}  [ \kappa_H'(\Re t) \nu_H(\Re t) - \kappa_H^{2}(\Re t) T_H(\Re t)] \, + {\mathcal O}( |\Im t |^{4}). \end{array} \end{equation}

\subsubsection{Taylor expansion near complex glancing set} \label{graph}
Then for  $| t - Y(s)| \leq \eo $ and $\eo >0$ sufficiently small,

\small{$$\\Re \rho^{\C}(t,s) = |q(Y(s))- r(s)| - (\Re t -Y(s)) \, \left( 1 + \frac{1}{2} \kappa_H^2(Y(s)) |\Im t|^{2} \right) +  \max_{\alpha + \beta = 4}  {\mathcal O}(|\Re t -Y(s)|^{\alpha} |\Im t|^{\beta}),$$}
\begin{equation} \label{nearglancing}
\Im \rho^{\C}(t,s) = - \Im t  + \frac{\kappa_H(Y(s))}{2} (\Re t -Y(s))^{2} \Im t - \frac{1}{6} \kappa_H^{2}(Y(s)) (\Im t)^{3} + \max_{\gamma + \delta = 5} {\mathcal O}(|\Re t -Y(s)|^{\gamma} |\Im t|^{\delta}).
\end{equation}

\normalsize

\begin{rem} \label{graph condition}
It follows from (\ref{nearglancing}) that  the real Lagrangian $\Re \Lambda \subset T^*\partial \Omega \times T^*H$ in (\ref{realwavefront}) is a canonical graph with respect to the symplectic form $\kappa_1^*(ds \wedge d\Im t) \oplus \kappa_2^*(-ds \wedge d\Im t)$ provided the tube radius $\eo >0$ is chosen sufficiently small. To see this, consider first the approximating Lagrangian $$ \Re \Lambda_{approx} := \{ (s, d_s \Re \rho^{\C}(t,s); \Re t, -d_{\Re t} \Re \rho^{\C}(t, s)); \,\, \Re t = Y(s),  \, \Im t \in (\frac{\eo}{2},\eo) \}.$$  Here, we write $t = (\Re t, \Im t)$ and identify  $\R^2$ with $\C$ in the usual way. 

Consider the associated parametrizing maps   $\kappa_1: (-\pi,\pi) \times (\frac{\eo}{2},\eo) \to \pi_{T^*\partial \Omega}(\Re \Lambda_{approx})$ with $$\kappa_1:(s, \Im t) \mapsto ( s, d_{s} \Re \rho^{\C}(t,s)  |_{\Re t = Y(s)} )$$ and  $\kappa_2: (-\pi,\pi) \times (\frac{\eo}{2},\eo) \to \pi_{T^*H}(\Re \Lambda_{approx}) $ with $$\kappa_2:(s, \Im t)  \mapsto ( Y(s), -d_{\Re t} \Re \rho^{\C}(t,s) |_{\Re t = Y(s)} ).$$ In view of (\ref{nearglancing}),
$$ \kappa_1(s,\Im t) = \Big( s,  d_s |q(Y(s)) - r(s)| + d_s Y(s) (1 + \frac{1}{2} \kappa^2_H(Y(s)) | \Im t|^2 ) + O(|\Im t|^3) \Big) ,$$
and
$$\kappa_2(s, \Im t) = \Big( Y(s),  1 + \frac{1}{2} \kappa^2_H(Y(s)) | \Im t|^2 + O(|\Im t|^3) \Big).$$
For $\eo >0$ sufficiently small, both maps $\kappa_j; j=1,2$ are diffeomorphisms  onto their images and consequently, so is $\kappa_2 \circ \kappa_1^{-1}.$ Finally, we note that from (\ref{exact2}), the constraint $s=s^*(t) = Y^{-1}(\Re t)  + O(|\Im t|)$ appearing in the definition of $\Re \Lambda$ in (\ref{realwavefront}) implies that  $Y(s) = \Re t + O(|\Im t|)$. It then follows by the above argument combined with the Implicit Function Theorem that for $\eo >0$ small, $\Re \Lambda$ is a canonical graph relative to the symplectic form $\kappa_1^{*} (ds \wedge d\Im t) \oplus \kappa_2^*(- ds \wedge d\Im t).$
\end{rem}






\section{Analysis of $N^{\C}(h)^*e^{-2S/h} a N^{\C}(h)$:  Proof of Proposition \ref{mainthm2prop} } \label{opanalysis}


In
this section,  we prove  Proposition \ref{mainthm2prop} by carrying out a careful analysis of the conjugate operator $N^{*}(h)aN(h): C^{\infty}(\partial \Omega ) \rightarrow C^{\infty}(\partial \Omega)$. This entails several complications, most important of which is that this operator is  only an $h$-pseudodifferential operator when $h$-microlocalized away from the glancing set (in the boundary case these are the tangential directions to the boundary). In quantum ergodicity or quantum ergodic restriction, these sets do not affect the limiting asymptotics and are therefore ignored \cite{TZ2}.
 However, here  the situation is very different. We are actually interested in the complexified operator $h^{-1/2}[e^{-S/h}  N^{\C}(h)]^* a [e^{-S/h} N^{\C}(h)]: C^{\infty}(\partial \Omega)\rightarrow C^{\infty}(\partial \Omega)$ where $a \in C^{\infty}_{0}(H_{\eo}^{\C})$ is supported in  $H_{2\eo/3}^{\C}-H_{\eo/6}^{\C}$ (see subsection \ref{mainthm2proof}).  In this case, as we have already pointed out in the introduction, it is precisely the glancing set $\Sigma$ and the corresponding glancing map $Y: C^{\infty}(\partial \Omega) \to C^{\infty}(H)$ that determine the leading operator asymptotics.  To analyze this operator, we will need to  make a further $h$-microlocal  decomposition by splitting the complex near-glancing directions into the ``near-real" and complementary directions.  Fortunately, the fact that we are  dealing with {\em complex} near-glancing sets (rather than real ones) actually simplifies the analysis of the microlocal complex $\eo$ near-glancing piece of the $N^{\C^*}(h) a N^{\C}(h)$-operator, as long as the support of $a \in C^{\infty}_{0}(H_{\eo}^{\C})$ lies outside an arbitrarily small neighbourhood of the real curve, $H$.
 The proof essentially consists of carrying out the details of $h$-microlocalisation indicated above and a key step is to show that one can make the decomposition (see section \ref{microlocal})
 
$$ h^{-1/2}[e^{-S/h}  N^{\C}(h)]^* a [e^{-S/h} N^{\C}(h)] = U_Y(h)^* T(h)^* a T(h) U_Y(h), $$

\bigskip
where, $h$-microlocally on supp $a$,  $T(h)$ is a generalized FBI transform. This is essentially the content of Lemma \ref{integralnf1}.



\begin{rem} Since $ | \langle \omega(s,Y(s)), T_{\partial \Omega}(s) \rangle | < 1,$ it follows from (\ref{correspondence}) and the support assumptions on $a(\Re t, \Im t)$ in (\ref{supportprop})  that for $\eo >0$ small, $a_{\gcal} \in C^{\infty}_{0}(B^*\partial \Omega)$ (ie. has support disjoint from the tangential set $S^*\partial \Omega$).
\end{rem}

\begin{proof}

We first cutoff near the glancing point $t= Y(s)$ by introducing a  cutoff function $\chi \in C^{\infty}_{0}(\C)$ with $\chi(z) =1$ when $|z| \leq \frac{\eo}{2}$ and $\chi(z) =0$ for $|z| \geq \eo$. Here, $\eo >0$ is fixed but chosen arbitrarily small.
We decompose the operator $N^{\C}(h)$ in various stages. First, we write
$$ e^{-S/h}N^{\C}(h) =   e^{-S/h}N_{1}^{\C}(h) + e^{-S/h}N_{2}^{\C}(h)  + {\mathcal E}(h)$$
where,
\begin{equation} \label{dec1}
e^{-S/h}N_{1}^{\C}( t, s;h) = C h^{-1} e^{i  [ \rho^{\C}(t, s) + i S(t)]/h}
\, \chi( |t -Y(s)| )  \,\,  b( h^{-1} \rho^{\C}(t,
s ) ),
\end{equation}
and
\begin{equation} \label{dec2}
e^{-S/h} N_{2}^{\C}( t, s;h) = C h^{-1} e^{i  [\rho^{\C}(t, s ) + i S(t)]/h}
\,(1- \chi)( |t -Y(s)| )  \,\,  b( h^{-1} \rho^{\C}(t,
s) ),
\end{equation}
where $b(t)$ has an asymptotic expansion in inverse powers of $t$ as
$t \to \infty$, with leading term $\sim t^{-1/2}$ and recall the glancing diffeomorphism $Y:[-\pi,\pi] \rightarrow [-\pi, \pi]$ is characterized by the identity $\langle T_{H}(Y(s)), \omega(s,Y(s)) \rangle = -1.$ The operator ${\mathcal E}(h): L^{2}(\partial \Omega) \rightarrow L^{2}(\partial \Omega)$ satisfies
$\| {\mathcal E}(h) \|_{L^2 \rightarrow L^2} = {\mathcal O}(h^{\infty})$ in view of the complex WKB expansion in (\ref{complexWKB}) and  is negligible.

From (\ref{dec1}) and (\ref{dec2}) we make the decomposition
\begin{align*}
&e^{-2S/h} N^{\C}(h)^{*}a N^{\C}(h)\\ 
&=  e^{-2S/h}N^{\C}_2(h)^{*} aN_{2}^{\C}(h) +  e^{-2S/h}N^{\C}_1(h)^{*}a N_{1}^{\C}(h) + e^{-2S/h} N^{\C}_2(h)^{*} aN_{1}^{\C}(h) + e^{-2S/h}N^{\C}_1(h)^{*}a N_{2}^{\C}(h).
\end{align*}
 We first analyze the diagonal terms $e^{-2S/h}N_{1}^{\C}(h)^{*} a N_{1}^{\C}(h)$ and $e^{-2S/h}N_{2}^{\C}(h)^{*}  a N_{2}^{\C}(h)$ and use  Cauchy-Schwarz to estimate the off-diagonal term $e^{-2S/h}N_{k}^{\C}(h)^{*} a N_{l}^{\C}(h)$ with $k \neq l$ at the end.

\subsection{Estimate for the $N_{2}^{\C *}(h) e^{-2S/h} a  N_{2}^{\C}(h)$-term} \label{fardiagonal}

This piece of $N^{\C}(h)^* e^{-2S/h} a N^{\C}(h)$ is easiest to control.  Indeed, since $a(t)$ is supported in an annular subset of $H_{\eo}^{\C}$ with supp $a \subset H_{2\eo/3}^{\C} - H_{\eo/6}^{\C},$ it follows from the asymptotic formula for the weight function $S(t)$ in Lemma \ref{weightlemma}  and the Taylor expansion in (\ref{taylor2}) that in the case where  $|Y(s) - t| \geq \eo$ and $|Y(s')-t| \geq \eo,$ there is a constant $C(\eo)>1$ such that
$$ \Im ( 2 S(t) + \rho^{\C}(t,s) - \rho^{\C}(t,s') ) \geq  2 |\Im t| - \frac{ 2 |\Im t|}{C(\eo)}+ {\mathcal O}(|\Im t|^2) \geq 2(1- C(\eo)^{-1}) |\Im t| + O(|\Im t|^2).$$ 
It follows  that  for $\eo >0$ small, there is a positive constant $C'(\eo)>0$ such that
\begin{equation} \label{expdecay}
 N_2^{\C}(h)^* e^{-2S/h} a N_{2}^{\C}(h) (s,s') = {\mathcal O}(e^{-C(\eo)/h}) \end{equation}
 uniformly for $(s,s',t) \in \partial \Omega \times \partial \Omega \times  ( H_{2\eo/3}^{\C} - H_{\eo/6}^{\C}).$ The same is true for the partial derivatives $ \partial_{s}^{\alpha} \partial_{s'}^{\beta}  [ N_2^{\C}(h)^* e^{-2S/h} a N_{2}^{\C}(h)] (s,s').$ 
\vspace{3mm}
\subsection{Estimate for $N^{\C*}_{1}(h)e^{-2S/h} a N_{1}^{\C}(h)$: The dominant term.} \label{dominant}

 \begin{align*}
& h^{-1/2} N^{\C *}_{1}  e^{-2S/h}a N^{\C}_{1}(s, s';h ) \\
&= h^{-1/2} \int\int_{S_{\eo,\pi}}  \chi (|t-Y(s')| ) \,\, \chi( |t-Y(s)| ) a(q^{\C}(t),\overline{q^{\C}(t)}) \, \overline{  N^{\C}(t, s';h)} \,\,  N^{\C}(t, s;h ) \,  e^{-2 S(t)h} \,\, dt d\overline{t},
 \end{align*}
 After decomposing  the resulting integral into two further pieces (depending on whether $|\Re t -Y(s)|$ or $|\Im t|$ dominates when $|t-Y(s)| \leq \eo$) and using the strict convexity of $H \subset \Omega$, we apply the method of steepest descent to expand the $\Re t$-integral. The remaining imaginary coordinate $\Im t$ then behaves roughly like a frequency variable in the oscillatory integral representation of an $h$-pseudodifferential operator of order zero (here,  we again use that $H$ is strictly convex).

In this case we  carry out the $\Re t$-integration first. We decompose the $h^{-1/2} N_{1}^{\C}(h)^* e^{-2S/h}a N_{1}^{\C}(h)$ operator further as follows: Let $\tilde{\chi} \in C^{\infty}_{0}(\C)$ be a cutoff equal to $1$  on a ball of radius $1/2$ and zero outside the ball of radius $1$. Also,   to simplify the writing, we abuse notation and write $a(t)$ for $a(q^{\C}(t),\overline{q^{\C}(t)})$ where the latter is supported in the strip (\ref{supportprop}) in the upper half plane. We define the operators $N^{11}_1(h)$ and $N_{1}^{22}(h)$ with Schwartz kernels
\begin{align} &N_{1}^{11}(h)(s,s'; a) \nonumber \\
\nonumber&= h^{-1/2} \int\int_{S_{\eo,\pi}} e^{-2S(t)/h} N_1^{\C}(h)^{*}(s,t) \, a(t) \, \tilde{\chi} \left( \frac{|\Re t - Y(s)|}{|\Im t|} \right) \,\, \tilde{\chi} \left( \frac{|\Re t - Y'(s)|}{|\Im t|} \right)  N_1^{\C}(h)(t,s') \, dt \, d\overline{t},
\end{align}

\begin{align} 
&N_{1}^{22}(h)(s,s'; a) \nonumber \\
\nonumber&= h^{-1/2} \int\int_{S_{\eo,\pi}} e^{-2S(t)/h} N_1^{\C}(h)^{*}(s,t) \, a(t) \,  ( 1- \tilde{\chi}) \left( \frac{|\Re t - Y(s)|}{|\Im t|} \right) \,( 1- \tilde{\chi}) \left( \frac{|\Re t - Y'(s)|}{|\Im t|} \right) \nonumber  \\
\nonumber&\,\,\,\,\,\,\,\,\,\,\,\,\,\,\,\,\,\,\,\,\,\,  \times N_1^{\C}(h)(t, s') \, d t \, d\overline{t},
\end{align}
\normalsize
and with mixed terms $N_{1}^{12}(h;a)$ and $N_{1}^{21}(h;a)$ defined in the obvious way
so that $$ h^{-1/2} N_1^{\C}(h)^{*}e^{-2S/h} a N_1^{\C}(h) = N_{1}^{11}(h;a) + N_{1}^{22}(h;a) + N_{1}^{12}(h;a) + N_{1}^{21}(h;a).$$

Just as before, we control the mixed terms $N_{1}^{12}$ and $N_{1}^{21}$ using Cauchy Schwarz and the estimates for the diagonal terms,  and it will  suffice to analyze the diagonal terms $N_{1}^{11}$ and $N_{1}^{22}$.

\subsubsection{ Analysis of the $N_1^{11}(h)$-term: Reduction to normal form.}
 Our aim here is  to reduce the operator $N_1^{11}(h)$  by suitable change of variables in $(\Re t, \Im t)$ to a normal form and then, by an application of  analytic stationary phase (\cite{Ho1} Theorem 7.7.12), we show that the normal form operator is in $Op_{h} ( S^{0,-\infty}(T^{*}\partial \Omega ) ).$

  First, we recall the asymptotic formulas for $\Re (i\Psi)$ and $\Im(i\Psi)$ where $\Psi$ is the phase function of Schwartz kernel of $h^{-1/2}e^{-2S/h} N^{\C}(h)^* a N^{\C}(h)$ given by
\begin{equation} \label{phase}
i \Psi(s,s',t)= i \rho^{\C}(t,s) - i \overline{\rho^{\C}(t,s)} - 2 S(t). \end{equation}
  From Lemma \ref{asymptotic},
\begin{align} \label{realphase}
\Re \,  [ i\Psi (t,s,s') ] &=   - \frac{\kappa_{H}^{2}(Y(s))}{2} (\Re t -Y(s))^{2} \Im t + \max_{\gamma + \delta = 4, \delta \geq 1} {\mathcal O}(|\Re t -Y(s)|^{\gamma} |\Im t|^{\delta})\notag\\
&- \frac{\kappa_{H}^2(Y(s'))}{2} (\Re t -Y(s'))^{2} \Im t + \max_{\gamma + \delta = 4, \delta \geq 1} {\mathcal O}(|\Re t -Y(s')|^{\gamma} |\Im t|^{\delta}).\notag\\
\end{align}
In (\ref{realphase}) we note that (see Lemma \ref{weightlemma})  the terms $\frac{1}{6} \kappa_{H}^{2}(Y(s)) (\Im t)^{3} + \frac{1}{6} \kappa_{H}^{2}(Y(s')) (\Im t)^{3} $ get cancelled by the cubic  term  in $\Im t$ term appearing in the expansion of $S(t)$ in Lemma \ref{weightlemma}.
Similarly,
\begin{align} \label{imphase}
&\Im [i \Psi(t,s,s')]\notag\\
& = |q(Y(s))- r(s)| - (\Re t -Y(s)) \, \left( 1 + \frac{1}{2} \kappa_{H}^2(Y(s)) |\Im t|^{2} \right)  + \frac{ \kappa^{2}_{H}(Y(s))}{6} (\Re t - Y(s))^{3} \notag \\
& +  \max_{\alpha + \beta = 4}  {\mathcal O}(|\Re t -Y(s)|^{\alpha} |\Im t|^{\beta}) \notag\\
& - |q(Y(s'))- r(s')| + (\Re t -Y(s')) \, \left( 1 + \frac{1}{2} \kappa_{H}^2(Y(s')) |\Im t|^{2} \right)  - \frac{ \kappa_{H}^{2}(Y(s'))}{6} (\Re t - Y(s'))^{3}\notag\\
&+  \max_{\alpha + \beta = 4}  {\mathcal O}(|\Re t -Y(s')|^{\alpha} |\Im t|^{\beta}).
\end{align}




Substitution of the identity $ \langle T_{H}(Y(s)), \omega(s,Y(s)) \rangle = - 1$ and second-order Taylor expansion around $s=s'$ in (\ref{imphase}) gives

\begin{align}\label{imphase2}
&\Im \, (i\Psi(t,s,s')) \nonumber \\
&=(Y(s)-Y(s')) \, \big( \langle T_{H}(Y(s)), \omega(s,Y(s)) \rangle - (\partial_{s} Y(s))^{-1} \langle T_{\partial \Omega}(s), \omega(s,Y(s))\rangle \big)\nonumber \\
&=(Y(s)-Y(s')) \, \left(  - \partial_s Y(s)^{-1} \langle T_{\partial \Omega}(s), \omega(s,Y(s))\rangle 
 + \frac{\kappa_{H}^{2}(Y(s))}{2} |\Im t|^{2} + {\mathcal O}(|\Im t|^{3}) \right)\nonumber \\
 & \,\,\,\,\,\,\,+ {\mathcal O}(|s-s'|^{2}) \nonumber \\
 &= (s-s') \, \left( -\langle T_{\partial \Omega}(s), \omega(s,Y(s))\rangle + d_s Y(s)   \frac{\kappa_{H}^{2}(Y(s))}{2} |\Im t|^{2} + {\mathcal O}(|\Im t|^{3})  \right)+ {\mathcal O}(|s-s'|^{2}).
\end{align}

For the error term in (\ref{imphase2}), we have used the constraints $\max (|\Re t -Y(s)|, |\Re t - Y(s')|) = {\mathcal O}(|\Im t|)$  and also note that the  ${\mathcal O}(|s-s'|^{2})$-term appearing on the RHS of (\ref{imphase2}) is independent of the $t$-variables since it comes from the second-order Taylor expansion of the real-valued function $|q(Y(s)) - r(s)|$ around $s= s'$.

\subsection{Normal form for the phase function $i \Psi(t,s,s')$} \label{normalform} We now reduce the computation of the  principal term $N^{11}_{1}(h)$ to a specific normal form by applying a series of changes of variables in the $(\Re t, \Im t)$-coordinates.
Given $(t,s) \in  S_{\eo,\pi} \times [-\pi,\pi]$  we claim that  near any point $ t_{0} \in S_{\eo,\pi}$ with $\eo >0$ sufficiently small, one can find a locally a real-valued analytic function $f (\Re t, \Im t)$  satisfying
\begin{equation} \label{weightift}
S(\Re t, \Im t) = \Re \, [ \ i  \rho^{\C}(t, Y^{-1} f(\Re t, \Im t)) \, ].
\end{equation}
with
\begin{equation} \label{critpoint}
f(\Re t, 0) = \Re t.
\end{equation}
To prove (\ref{weightift}) and also (\ref{critpoint}), consider the real analytic function $ g\in C^{\omega}( S_{\eo,\pi} \times [-\pi,\pi])$ defined by
\begin{equation} \label{gfunction}
g(t,s) := \frac{ \Re \, i  \rho^{\C} (t,s)}{ \Im t}.
\end{equation}
We wish to solve
\begin{equation} \label{weightift2}
\partial_{s} \, g(t, s) =0,
\end{equation}
where from the Taylor expansion in (\ref{realphase}), there is the initial condition
 $$ \partial_{s} g(t,s)|_{\Re t=Y(s), \Im t = 0} = 0.$$
 Thus, (\ref{weightift}) follows from the Implicit Function Theorem applied to (\ref{weightift2}),  since from (\ref{realphase}) and the  strict convexity of $H,$ (ie. $\kappa_{H} >0$) we get that for $\eo >0$ sufficiently small,
$$\partial_{s}^{2} g(t,s)  \leq - \kappa_{H}^2(Y(s)) |Y'(s)|^2 \,  + {\mathcal O}( \eo^{2})$$
 since $|\Re t-Y(s)| \leq \eo$ and $|\Im t| \leq \eo$.  Then, $s=Y^{-1}f(t)$ is a local maximum for $g(t, \cdot )$ and by definition of $S(t)$, it is clear that
\begin{equation} \label{control1}
\Re \, [i\Psi ](t,s,s')] \leq 0.
\end{equation}
Given (\ref{weightift}) we also have for all $t \in S_{\eo,\pi},$
\begin{equation} \label{control2}
\Re \, [  i\Psi (t; Y^{-1}f(t), Y^{-1}f(t)) ]  = 0.
\end{equation}
Equations (\ref{control1}) and (\ref{control2}) show that $s=Y^{-1}f(t)$ is, in fact, a  {\em global } maximum for $g(t, \cdot ).$ 
Consequently, we note that 
$$Y^{-1}f(t) = s^*(t)$$ in the notation of 
Lemma \ref{weightlemma}. 
To simplify notation in (\ref{control2}) and the following, we identify the complex variable $t$ with the real 2-tuple $(\Re t, \Im t)$ in the argument of $i \Psi$. The pair of coordinates $(Y^{-1}f(t), Y^{-1}f(t))$ occupy the $(s,s')$ coordinate slots.
 By definition $S(t) = \max_{s \in [0,2\pi]} \Re i \rho^{\C}(t,s)$ so that  $\partial_{s} \Re \,  [ i  \rho^{\C}(t, s) ] |_{s=Y^{-1}f(t)} =0.$ By differentiating (\ref{control2}) in $\Re t$ it follows that for any $t \in S_{\eo,\pi},$
 \begin{equation} \label{realcrit}
 \partial_{\Re t} \, \Re  [i \Psi] (t,s^*(t),s^*(t)) =0.
 \end{equation}
 Since  $\Im i \Psi(t,s,s) =0$, the identity $\partial_{\Re t}  \, [ \Im i \Psi ] (t,s^*(t),s^*(t))=0$ is automatic  and so,
\begin{equation} \label{control3}
\partial_{\Re t} \,  [ i\Psi] (t;s^*(t),s^*(t)) \, = \, 0.
\end{equation}
Since   $i\Psi(t,s,s') \in C^{\omega} (S_{\eo,\pi} \times \R^2/(2\pi \Z)^2 )$ and $H \subset \Omega $ is strictly convex,      from (\ref{realphase}),
$$ |\partial_{\Re t}^{2} (i\Psi)| \geq \, \left[ \kappa_{H}^{2}(Y(s)) + \kappa_{H}^{2}(Y(s'))  \right] \, |\Im t| + {\mathcal O}( \eo^2 |\Im t|  ) \geq  C(\eo) |\Im t|,$$
where $C(\eo) >0$ with $\eo >0$  sufficiently small. Differentiating (\ref{realcrit}) yet again in $\Re t$ gives
\begin{equation} \label{hessian}
 \partial_{\Re t}^{2} \Re [i\Psi] (t,s^*(t)) + 2  \, \partial_{s} \partial_{\Re t} \Re [i\Psi](t,s^*(t),s^*(t)) \cdot \, \partial_{\Re t} s^*(t)  = 0.
 \end{equation}
In view of the Taylor expansion (\ref{realphase}), this simplifies to
$$  2 \kappa_{H}(\Re t)^{2} |\Im t| + {\mathcal O}(|\Re t - f(t)| |\Im t|) + {\mathcal O}(|\Im t|^{2})   -  \, 2 \kappa_{H}(\Re t)^{2} |\Im t| \,  \cdot \partial_{\Re t}f(t) = 0.$$
By dividing the last equation through by $\kappa_{H}(\Re t)^{2} \Im t,$ and solving for $\partial_{\Re t}f,$  one gets that  $f(t) = \Re t + {\mathcal O}(|\Im t|).$ By the same identity,

\begin{equation} \label{critpoint2}
\partial_{\Re t} f (t) = 1 + {\mathcal O}(|\Im t|).
\end{equation}

Given $(\ref{critpoint2})$, we make the change of  variables $(\Re t, \Im t) \mapsto (f(t), \Im t)$ in the tubular parameter domain $S_{ \delta(\eo),2\pi}$ with $\delta(\eo) >0$ sufficiently small. To reduce to normal form for $i\Psi$ we define the new variable
\begin{align} \label{realphase1}
 \tau_{1} (\Re t, \Im t) &= f(t) = \Re t ( 1+ {\mathcal O}(\Im t)), \nonumber \\
 \partial_{\Re t} \tau_1 &= 1 + {\mathcal O}(\Im t).\end{align}

The complementary variable $\tau_2$ is defined by writing
\begin{equation} \label{imphase1}
\Im [i \Psi](t,s,s') = (Y(s)-Y(s')) \cdot \tau_2,
\end{equation}
where from (\ref{imphase2})  we know that
\begin{align} \label{imphase3}
 \tau_2 &= - (d_sY(s))^{-1} \langle T_{\partial \Omega}(s), \omega(s,Y(s))\rangle +   \frac{\kappa_{H}^{2}(Y(s))}{2} |\Im t|^{2} + {\mathcal O}(|\Im t|^{3}), \notag \\
\partial_{\Im t} \tau_{2} &= \kappa_H^2(Y(s)) \Im t + {\mathcal O}(|\Im t|^2),  \end{align}

\bigskip
since the last error term in (\ref{imphase2}) is independent of the $t$-variables.  
 Since $a(t) \in C^{\infty}_{0}( \{ t; \frac{\eo}{6}  <  |\Im t| < \frac{2\eo}{3} \})$ it follows from (\ref{imphase3}) and the derivative computation in $(\ref{imphase3})$ that for the change of variables
 $ (\Re t, \Im t) \mapsto (\tau_1, \tau_2),$
 \begin{equation} \label{jacobian}
 J(t;s,s'):= \left| \frac{\partial(\tau_1,\tau_2)}{\partial(\Re t, \Im t) } \right|  = \kappa_H^{2}(Y(s)) \cdot |\Im t| + {\mathcal O}(|\Im t|^{2})  \geq C(\eo) |\Im t| >0, \,\,\, t \in \, \text{supp} \,a.
 \end{equation}
We note that the positive curvature of $H$ is used at this point in carrying out the change of variables $(\Re t, \Im t) \mapsto (\tau_1,\tau_2).$

From (\ref{realphase1}) and (\ref{imphase1}) we derive the following normal form for the phase function $i\Psi.$ 


\begin{lem} \label{normalform}
In terms of the new coordinates $(\tau_1, \tau_2)$  in $H_{\eo}^{\C}$ defined in (\ref{realphase1}) and (\ref{imphase1}), it follows that the real part
\small{$$ \Re [i \Psi]( t(\tau_1,\tau_2); s, s') =  - \alpha (\tau_1, \tau_2)  \left[ | Y(s) - \tau_1 | ^{2} + | Y(s') - \tau_1 | ^{2} + {\mathcal O}(|Y(s) - \tau_1 |^{3})  + |Y(s') - \tau_1|^{3}) \right],$$}
\normalsize
where,  $\alpha(\tau_1, \tau_2) =[ \, \kappa_{H}^{2}(t(\tau_1,\tau_2)) + {\mathcal O}(\eo) \, ] \Im t.$
\medskip

The imaginary part
$$ \Im [i \Psi](t(\tau_1,\tau_2);s,s') = (Y(s)-Y'(s)) \tau_2,$$
where,
$$\tau_2 (t(\tau_1,\tau_2); s,s') = - \partial_s Y(s)^{-1} \langle T_{\partial \Omega}(s), \omega(Y(s),s)\rangle
 + \frac{\kappa^{2}(Y(s))}{2} |\Im t|^{2} + {\mathcal O}(|\Im t|^{3}) + {\mathcal O}(|s-s'|).$$
\end{lem}
\begin{proof}
The formula for $\Re(i\Psi)$ follows from the Taylor expansion in (\ref{realphase}) plus the formula for the second derivative in (\ref{hessian}). The formula for $\Im (i\Psi)$ inturn follows from (\ref{imphase2}).
\end{proof}



We summarize our analysis so far in the following
\begin{lem} \label{integralnf1}
Let  $a \in C^{\infty}_{0}(H_{2\eo/3}^{\C} - H_{\eo/6}^{\C}) $  Then for $
\eo>0$ small enough  and $h \in (0, h_{0}(\eo)],$ the kernel $N_{1}^{11}(h)(s,s';a)$ equals
$$(2\pi h)^{-3/2} \int_{\R} \int_{\R}  \exp \, [ i \, (Y(s)-Y(s')) \tau_{2} - \beta_1 (Y(s)- \tau_1)^{2} -\beta_2 (Y(s')- \tau_1)^{2}] / h $$
$$ \times a( t(\tau_1,\tau_2) ) \,\, \tilde{\chi} \left( \frac{|\Re t - Y(s)]|}{|\Im t|} \right)  \tilde{\chi} \left( \frac{|\Re t - Y(s')]|}{|\Im t|} \right) \times \rho (\tau_1, \tau_2;s,s';h') \,  d\tau_1 \, d\tau_2.$$
Here,
\begin{align*}
 &\rho(\tau_1,\tau_2;s,s';h) = J^{-1}(t(\tau);s,s') \, b(t(\tau);s,s';h)\\
 &= \frac{ \kappa^{-2}_{H}(Y(s))}{\sqrt{2}}  \, | \Im t(\tau) |^{-1}
   \,  b(t(\tau);s,s';h) \, ( 1 + {\mathcal O}(|\Im t|),
\end{align*}
and $b \sim  \sum _{j=0}^{\infty}b_{j} h^j$  with
\begin{align*}
&b_0(t(\tau);s,s') = \langle \nu_{Y(s)}, \omega(Y(s),t(\tau)) \rangle \, \langle \nu_{Y(s')}, \omega(Y(s'),t(\tau)) \rangle
\end{align*}
and $J(t;s,s')$ is the Jacobian in (\ref{jacobian}). Here, $\beta_1(\tau_1,\tau_2,s) = \alpha_1(\tau_1,\tau_2) + {\mathcal O}(|Y(s)-\tau_1|)$ and $\beta_2(\tau_1,\tau_2,s') = \alpha_1(\tau_1,\tau_2) + {\mathcal O}(|Y(s')-\tau_1|)$ with $\beta_1 = \beta_2 + {\mathcal O}|s-s'|)$.
\end{lem}
\begin{proof}  Let  $\Psi(t(\tau_1,\tau_2),s,s')$ be the phase function in (\ref{normalform}). Then, with any fixed $\delta \in [0,1),$
\begin{align} \label{egorov1}
\nonumber &N_{1}^{11}(h)(s,s';a) \\ \nonumber
&= (2\pi h)^{-3/2} \int_{\R} \int_{\R} e^{i \Psi(t(\tau_1,\tau_2),s,s')/h} \,\,a(t(\tau_1,\tau_2)) 
\  \tilde{\chi} \left( \frac{|\Re t - Y(s)]|}{|\Im t|} \right)  \tilde{\chi} \left( \frac{|\Re t - Y(s')]|}{|\Im t|} \right) \\
& \,\ \, \times  \rho (\tau_1, \tau_2;s,s';h') \, d\tau_1 \, d\tau_2 + O(h^{\infty}) \end{align}
and the remainder in (\ref{egorov1}) is uniform for $(s,s') \in [-\pi,\pi]\times [-\pi,\pi].$

Since $\tau_1 = \Re t (1+ O(|\Im t|))$ and supp $a \subset \{ t; \frac{\eo}{6} \leq |\Im t| \leq \frac{2\eo}{3} \},$ the amplitude in (\ref{egorov1}) is supported  near the diagonal $s=s'$ where
$$ \max ( |Y(s)- \tau_1|, |Y(s') - \tau_1| ) \lessapprox \eo.$$
By possibly shrinking $\eo >0,$
 Lemma \ref{integralnf1} follows from Lemma \ref{normalform} after using the Morse lemma to make another change of variables of the form $\tau_1  \mapsto \tau_1 + {\mathcal O}( |\tau_1 - Y(s)|^{2} + |\tau_1 - Y(s')|^{2})$ in (\ref{egorov1}). To simplify notation, we continue to denote the new coordinate by $\tau_1$.  \end{proof}

To further simply the kernel in Lemma \ref{integralnf1}, we cutoff to the sets where,   $|s-s'| \leq \eo$ and $\max (|Y(s)- \tau_1|, |Y(s')-\tau_1|) \leq \eo$

 First, note that since
$$ \Im \partial_{\tau_2} [ i \, (Y(s)-Y(s')) \tau_{2} - \beta_1 (Y(s)- \tau_1)^{2} - \beta_2 (Y(s')- \tau_1)^{2}] = i ( Y(s)- Y(s'))$$

 and $|s-s'| \lessapprox |Y(s)-Y(s')| \lessapprox |s-s'|,$ it follows by repeated integration by parts in $\tau_2$ that

 \begin{align} \label{integralnf2}
&N_{1}^{11}(h)(s,s';a) \notag\\
=&(2\pi h)^{-3/2} \int_{\R} \int_{\R}  \exp \, [ i \, (Y(s)-Y(s')) \tau_{2} - \beta_1 (Y(s)- \tau_1)^{2} - \beta_2 (Y(s')- \tau_1)^{2}] / h \notag \\
&\times  a( t(\tau_1, \tau_2) ) \, \tilde{\chi} \left( \frac{|\tau_1- Y(s)|}{|\Im t|} \right)  \tilde{\chi} \left( \frac{|\tau_1 - Y(s')|}{|\Im t|} \right)  \rho(\tau_1, \tau_2;s,s';h') \, \chi(\eo^{-1}(s-s'))  d\tau_1 \, d\tau_2 \nonumber\\
&+ {\mathcal O}(h^{\infty}).\notag \\
\end{align}

Next, we note that under the curvature assumption $\kappa_H >0,$   the functions $\beta_j \gtrapprox \Im t \gtrapprox \eo$  for $t \in \text{supp} \, a.$ Consequently,
$$ \beta_1 (Y(s)- \tau_1)^{2} +  \beta_2 (Y(s')- \tau_1)^{2} \gtrapprox |Y(s)-\tau_1|^2 + |Y(s')-\tau_1|^2.$$
 As result, it follows that for any fixed $\delta \in (1/2,1),$

\begin{align} \label{integralnf2.5}
&N_{1}^{11}(h)(s,s';a) \notag\\
=&(2\pi h)^{-3/2} \int_{\R} \int_{\R}  \exp \, [ i \, (Y(s)-Y(s')) \tau_{2} - \beta_1 (Y(s)- \tau_1)^{2} - \beta_2 (Y(s')- \tau_1)^{2}] / h  \,\,\, a( t(\tau_1, \tau_2) ) \,\notag \\
&\times  \chi ( \eo^{-1} |\tau_1- Y(s)| ) \,  \chi ( \eo^{-1} |\tau_1 - Y(s')|) \, \rho(\tau_1, \tau_2;s,s';h') \, \chi( \eo^{-1} (s-s'))  d\tau_1 \, d\tau_2 + {\mathcal O}(h^{\infty}).\notag \\
\end{align}

In (\ref{integralnf2.5}), the cutoffs $\tilde{\chi} \left( \frac{|\tau_1- Y(s)|}{|\Im t|} \right)  \tilde{\chi} \left( \frac{|\tau_1 - Y(s')|}{|\Im t|} \right)$ have been removed since they are now redundant in view of the $\eo$-cutoffs.

Using the fact that $\beta_1 = \beta_2 + O(|s-s'|)$, and in view of the diagonal cutoff $\chi(\eo^{-1}(s-s')),$ it follows by Taylor expansion around $s'=s,$ that
$$ e^{- [ \beta_1 (Y(s)- \tau_1)^{2} + \beta_2 (Y(s')- \tau_1)^{2}] / h } = e^{- \beta_1 [  (Y(s)- \tau_1)^{2} +  (Y(s')- \tau_1)^{2}] / h } \, \left( 1 + O ( \, h^{-1}|Y(s')-\tau_1|^2 (s-s') \, ) \right)$$
uniformly in $(s,s')$ and substitute this expansion in (\ref{integralnf2.5}). 


Next, we consider the iterated $\tau_1$ Laplace integral
\begin{equation} \label{iterated1}
I_{1}(s,s',\tau_2;h) := \int_{\R} e^{- \beta_1 [(Y(s)-\tau_1)^{2} + (Y(s') -\tau_1)^{2}]/h} \tilde{\rho}(\tau_1,\tau_2;s,s',h)\, a(t(\tau_1,\tau_2)) \, d\tau_1,
\end{equation}
where,
$$ \tilde{\rho}(\tau_1,\tau_2;s,s',h) =  \rho(\tau_1,\tau_2;s,s',h)  \, \left( 1 + O ( \, h^{-1}|Y(s')-\tau_1|^2 (s-s') \, ) \right)$$
$$\times  \chi ( \eo^{-1} |\tau_1- Y(s)| ) \,  \chi ( \eo^{-1} |\tau_1 - Y(s')|) \, \chi(\eo^{-1} (s-s')). $$
The critical points of the phase are 
$$ \tau_{1,c}(s,s') = \frac{Y(s) + Y(s')}{2} + {\mathcal O}(|s-s'|^{2}).$$

Consequently, by steepest descent in $\tau_1,$ it follows that  for $h$ sufficiently small,

\begin{align} \label{iterated2}
I_{1}(s,s',\tau_2;h) &= (2\pi h)^{1/2} e^{-  \beta_1 \, [ |Y(s)-Y(s')|^{2} + {\mathcal O}(|s-s'|^3) \, ]/h } \, a ( t(\tau_{1,c}, \tau_2) ) \, [ \tilde{\rho} (\tau_{1,c},\tau_2,s,s',h) + O(h) ]. \end{align}
In (\ref{iterated2}) and below we abuse notation somewhat and simply write $\tau_{1,c}$ for $\tau_{1,c}|_{s=s'}.$
In view of the diagonal cutoff $\chi(\eo^{-1}(s-s'))$, Taylor expansion of the exponential in (\ref{iterated2}) gives
$$ e^{-  \, \beta_1 [ |Y(s)-Y(s')|^{2} + {\mathcal O}(|s-s'|^3)]/h } = 1 + O( h^{-1}|s-s'|^2).$$

 Consequently,
\begin{align} \label{iterated2.2}
I_{1}(s,s',\tau_2;h) =  (2\pi h)^{1/2} a( t (\tau_{1,c}, \tau_2) ) \, [ \tilde{\rho}(\tau_{1,c},\tau_2;s,s;h)  + O(h)]\,  (1  + O( h^{-1}|s-s'|^2) ). \end{align}

Since $ h^{-1}|Y(s')-\tau_{1,c}|^2 (s-s') =  O(h^{-1} |s-s'|^3),$ substitution of  (\ref{iterated2.2}) in (\ref{integralnf2}) gives
\begin{align} \label{integralnf3}
N^{11}_{1}(h)(s,s';a) & = (2\pi h)^{-1} \int_{\R} e^{i(Y(s)-Y(s')) \tau_2/h} \, a(t(\tau_{1,c}, \tau_2))  \nonumber \\
&\times ( \rho(\tau_{1,c},\tau_2,s,s;h)  + O(h) )\,  ( 1 + {\mathcal O}(h^{-1}|s-s'|^2 ) )  \, \chi( \eo^{-1}(s-s')) \, d\tau_2.   \end{align}

\vspace{3mm}
Make the change of variables $\tau_2 \to Y'(s) \tau_2=: \sigma$ in (\ref{integralnf3}). By integrating by parts in $\sigma$ and Taylor expansion of the amplitude around $s=s'$, one gets the formula

\begin{equation} \label{integralnf4}
N^{11}_{1}(h)(s,s';a) = (2\pi h)^{-1} \int_{\R} e^{i(s-s') \sigma/h} \,  a (t (Y(s), d_sY(s)^{-1} \sigma)  )  \end{equation}

$$ \times ( 1 + {\mathcal O}(h)) \, \rho (Y(s), d_sY(s)^{-1}\sigma,s,s;h ) \, |d_{s}Y(s)|^{-1} d\sigma. $$

\subsubsection{Identification of $H_{\eo}^{\C}$ with a subdomain of $B^*\partial \Omega$}
We collect here the explicit formulas identifying $H_{\eo}^{\C} = q^{\C}(S_{\eo,\pi})$ with a subset of $B^*\partial \Omega.$ Specifically, given $(\Re t, \Im t) \in S_{\eo,\pi},$ it follows from (\ref{imphase2}) that for the frequency variable $ \sigma \in B_s^*\partial \Omega,$
\begin{align} \label{ballfrequency}
\sigma &=  \partial_{s'} \Im [i\Psi](\Re t, \Im t; s,s')|_{s'=s} \nonumber \\
&= - \langle T_{\partial \Omega}(s), \omega(s,Y(s)) \rangle  + d_s Y(s)   \frac{\kappa_{H}^{2}(Y(s))}{2} |\Im t|^{2} + {\mathcal O}(|\Im t|^{3}).
\end{align}
As for the spatial variable $s \in [-\pi,\pi],$ from (\ref{realphase1}),
\begin{equation} \label{ballspatial}
Y(s) = f(\Re t , \Im t) = \Re t ( 1 + {\mathcal O}(|\Im t|)).\end{equation}
Again, from (\ref{ballfrequency}) it is clear that
$ |\sigma|  <1$ when $\eo >0$ is sufficiently small.

\begin{defn} \label{glancing symbol}
 We define the {\em glancing  symbol relative to $H$} associated with $a(\Re t, \Im t) \in C^{\infty}_{0}(H_{\eo}^{\C})$ to be $a_{\gcal}(s,\sigma) \in C^{\infty}_{0}(\partial \Omega)$ with
\begin{equation} \label{grazingsymbol}
a_{\mathcal G} (s,\sigma) := a(\Re t(Y(s),\sigma ),\Im t(Y(s),\sigma) ) \times \rho(Y(s), d_sY(s)^{-1} \sigma,s,s;0), \end{equation}
where $\rho$ is the function given in Lemma \ref{integralnf1}.
\end{defn}

Then, from (\ref{integralnf4}),  by $L^{2}$-boundedness and the fact that  $|b_0(Y(s),\Re t(Y(s),\sigma))|^2  = (1- |\sigma|^2),$   we have
\begin{equation} \label{lead}
 N^{11}_{1}(h; a) =  Op_{h}(a_{\mathcal G}) + {\mathcal O}(h)_{L^{2} \rightarrow L^{2}}.
\end{equation}

Moreover, since $\langle \nu_{Y(s)}, \omega(Y(s),\Re t) \rangle = \gamma(Y(s),  \sigma)$ it follows  from Lemma \ref{integralnf1} that

\begin{equation} \label{symbolformula}
a_{\gcal}(s,\sigma) = \frac{1}{\sqrt{2}} a(\Re t(Y(s),\sigma), \Im t(Y(s),\sigma) ) \,\kappa_H^{-2}(Y(s)) \,  |\Im t(Y(s),\sigma)|^{-1} \, \gamma^2(Y(s), \sigma).\end{equation}
Thus, from (\ref{ballfrequency}) and (\ref{ballspatial}), it follows that $ a_{\gcal} \in C^{\infty}_{0}(B^*\partial \Omega)$ with
\begin{equation} \label{elliptic}
a_{\gcal}(s,\sigma)  \geq \frac{1}{C} >0 \end{equation}
 when $(\Re t(Y(s),\sigma),\Im t(Y(s),\sigma))  \in \text{supp} \, a  \subset \{ q^{\C}(t) \in H_{\eo}^{\C}; \frac{\eo}{6} < |\Im t| < \frac{2\eo}{3} \}.$

\subsubsection{Analysis of the $N_{1}^{22}(h;a)$-term} We now estimate the contribution to  $N^{\C}_{1}(h)^* e^{-2S/h} a N_{1}^{\C}(h)$ coming from $N^{22}_{1}(h;a)$ where, we recall that
\begin{align} \label{subdivide2}
&N_{1}^{22}(h)(s,s'; a) \\
\nonumber& = (2\pi h)^{-3/2}\int  \int_{S_{\eo,\pi}} N_1^{\C}(h)^{*}(Y(s),t) \,\, a(t) \,\,  ( 1- \tilde{\chi}) \left( \frac{|\Re t - Y(s)|}{|\Im t|} \right) \,( 1- \tilde{\chi}) \left( \frac{|\Re t - Y(s')|}{|\Im t|} \right) \notag \\
\nonumber&\,\,\,\,\,\times N^{\C}(h)(t, Y(s')) \, d t \, d\overline{t}\\
\nonumber&= (2\pi h)^{-3/2} \int \int_{S_{\eo,\pi}} e^{i\Psi(s,s',t)/h} \,\, a(t) \,\,  ( 1- \tilde{\chi}) \left( \frac{|\Re t - Y(s)|}{|\Im t|} \right) \,( 1- \tilde{\chi}) \left( \frac{|\Re t - Y(s')|}{|\Im t|} \right) \, dt d\overline{t}.\notag
\end{align}
So, in this case we restrict to the range
\begin{equation} \label{constraint 1}
|\Im t| \leq \min ( |\Re t -Y(s)|, |\Re t - Y(s')|) \end{equation}
in the Taylor expansions (\ref{realphase}) and (\ref{imphase}) of the phase function, $i\Psi(t,s,s')$. In addition, we have the constraint (coming from the definition of $N_1^{\C}(h)$ in (\ref{dec1}) ) that
\begin{equation} \label{constraint 2}
\max ( |t-Y(s)|, |t-Y(s')| ) \lessapprox \eo. \end{equation}
  Then, for $\Im t \geq 0,$
\begin{equation} \label{constraint 3} \begin{array}{ll}
 \Re i \Psi (t,s,s') = - \left[ \, \kappa_{H}(Y(s))^{2}( \Re t  - Y(s) )^{2}  + \kappa_{H}(Y(s'))^{2} (\Re t - Y(s'))^{2} \, \right] \, \Im t\\ \\
 + {\mathcal O}(|\Re t - Y(s)|^{\alpha} |\Im t|^{\beta}) + {\mathcal O}(|\Re t - Y(s')|^{\alpha} |\Im t|^{\beta}) \end{array} \end{equation}

 where, $\alpha + \beta \geq 4$ and $\beta \geq 1$.
 Substituting the constraints in $(\ref{constraint 1})$ and $(\ref{constraint 2})$ in $(\ref{constraint 3})$ implies that
 \begin{equation} \label{phasebound}
 \Re [i \Psi](t,s,s') \leq  - 2 \kappa_{H}^{2} |\Im t|^{3} + {\mathcal O}(\eo) |\Im t|^{3} \leq - C(\eo) |\Im t|^{3},\end{equation}
 with $C(\eo) >0$ provided $\eo >0$  is sufficiently small.
Substitution of the phase bound in (\ref{phasebound})  in the Schwartz kernel formula in (\ref{subdivide2}) gives
 \begin{align*}
 &|N_{1}^{22}(h)(y,y',a)| \\
 &\small{\leq (2\pi h)^{-3/2} \int \int_{S_{\eo,\pi}} e^{\Re [i\Psi](t,s,s')/h} \, |a(t)| \,  ( 1- \tilde{\chi}) \left( \frac{|\Re t - Y(s)|}{|\Im t|} \right) \,( 1- \tilde{\chi}) \left( \frac{|\Re t - Y(s')|}{|\Im t|} \right) \, dt d\overline{t}}\\
&\leq (2\pi h)^{-3/2} \int \int_{S_{\eo,\pi}} e^{ - C(\eo) |\Im t|^{3} /h}  \, |a(t)| \, dt d\overline{t}\\
& = {\mathcal O}(h^{-3/2} e^{-C(\eo) \, \eo^{3}/h}),
\end{align*}
\normalsize
since by assumption supp $a \subset H_{\eo}^{\C} \cap \{t; \frac{\eo}{6} < \Im t < \frac{2\eo}{3} \}.$ Consequently, the $N_{1}^{22}(h)$-term is exponentially decaying in $h$ and is  negligible.


\subsection{Mixed terms}
In the following, we continue to write $L^2:= L^{2}(\partial \Omega).$  Then, in view of the analysis in section \ref{normalform}  there is the decomposition $h^{-1/2}N_{1}^{\C}(h)^* e^{-2S/h} a N_{1}^{\C}(h) = N_{1}^{11}(h;a) + N_1^{22}(h;a) + N_{1}^{21}(h;a) + N_{1}^{12}(h;a)$, where
$$ N_{1}^{11}(h;a) =  Op_{h}(a_{\mathcal G}) + {\mathcal O}(h)_{L^2 \rightarrow L^2},$$
$$ \| N_{1}^{22}(h;a) \|_{L^2 \rightarrow L^2} = {\mathcal O}(h^{-1} e^{-C(\eo)  /h}),$$
with $C(\eo) >0$ and $$ \| N_{1}^{12}(h;a)^{*} N_{1}^{12}(h;a) \|_{L^{2} \rightarrow L^{2}}  = \| N_{1}^{11}(h;a)^{*}  N_{1}^{22}(h;a) \|_{L^{2} \rightarrow L^2} = {\mathcal O}(h^{-1/2} e^{-C(\eo)/h}).$$

The same estimate holds for $N_{1}^{21}(h;a)^{*} N_{1}^{21}(h;a)$.
As a result,
\begin{equation} \label{lead2}
h^{-1/2}N_{1}^{\C}(h)^*e^{-2S/h} a N_{1}^{\C}(h)  =  Op_{h}(a_{\mathcal G})  + {\mathcal O}(h).
\end{equation}
So, in particular,
\begin{equation} \label{n1}
\| h^{-1/4}e^{-S/h} N_{1}^{\C}(h) \|_{L^2(\partial \Omega) \rightarrow L^2(\text{supp} \,a)} = {\mathcal O}(1). \end{equation}
From the ``far-diagonal" bound in (\ref{expdecay})
$$ \| N_{2}^{\C}(h)^{*} e^{-2S/h} a N_{2}^{\C}(h) \|_{L^2 \rightarrow L^{2}} = {\mathcal O}(e^{-C(\eo)/h}).$$
Thus,
\begin{equation} \label{n2}
 \| e^{-S/h} N_{2}^{\C}(h) \|_{L^2(\partial \Omega) \rightarrow L^2(\text{supp}\, a)} = {\mathcal O}(e^{-C(\eo)/2h}).\end{equation}
 From (\ref{n1}) and (\ref{n2}) it then follows by Cauchy-Schwarz that the mixed terms
 \begin{equation} \label{mixed}
\max (  \| N_{2}^{\C}(h)^* e^{-2S/h} a N_{1}^{\C}(h) \|_{L^2 \rightarrow L^2}, \| N_{1}^{\C}(h)^* e^{-2S/h} a N_{2}^{\C}(h) \|_{L^2 \rightarrow L^2}  ) = {\mathcal O}( e^{-C'(\eo)/h}) \end{equation}
with $C'(\eo) >0.$
 So,  (\ref{lead2}) and (\ref{mixed}) imply that for $h \in (0, h_0(\eo)]$ with $h_0 >0$ sufficiently small,
$$ h^{-1/2} N^{\C}(h)^* e^{-2S/h} a N^{\C}(h)  = Op_{h}(a_{\mathcal G})  + {\mathcal O}(h)_{L^2 \rightarrow L^2}.$$
This finishes the proof of Proposition \ref{mainthm2prop}. \end{proof}

\section{Analysis near corner points} \label{corners}

We assume now that  $\Omega \subset \R^2$ is a smooth domain  with corners. We define a smooth domain
with corners in $\R^n$ and with $M$ boundary faces (hypersurfaces) to be a
set of the form $\{x \in \R^n: \rho_j(x) \leq 0, j = 1, \dots, M\}$,
where the defining functions $\rho_j$ are smooth in a
neighborhood of $\Omega$ with $d \rho_j |_{\rho_j^{-1}(0)} \neq 0.$  A boundary hypersurface
$H_j$ is the intersection of $\Omega$ with one of the
hypersurfaces $\{\rho_j = 0\}$ The intersections of the boundary faces, $H_i \cap H_j,$ consist of finitely-many corner points. In addition, we require that $\partial \Omega$ is a Lipschitz boundary \cite{ZZ} (ie. locally given by a graph of a Lipschitz function).  In Theorem \ref{mainthm2} it is essential to allow $\Omega$ to have corners since domains
with ergodic billiard flow in $\R^n$ are non-smooth. 

We denote the smooth   part of $\partial
\Omega$ by $(\partial \Omega)^o  $.  Here, and throughout this
article, we denote by $W^o$ the interior of a set $W$ and, when no
confusion is possible, we also use it to denote the regular set of
$\partial \Omega$. Thus, $\partial \Omega = (\partial \Omega)^o
\cup \Sigma$, where $\Sigma = \bigcup_{i \neq j} (W_i \cap W_j)$
is the singular set.
 When $\dim \Omega = 2,$ the singular set is a
finite set of points and the $W_i$ are smooth curves. In
higher dimensions, the $W_i$ are smooth hypersurfaces; $W_i
\cap W_j$  is a stratified smooth space of co-dimension one, and
in particular $\Sigma$ is of measure zero. We denote by
$S^*_{\Sigma}\Omega$ the set of unit vectors to $\Omega$ based at
points of $\Sigma.$
 We also
define $C^\infty(\partial \Omega)$ to be the restriction of
$C^\infty(\R^n)$ to $\partial \Omega$. We define the open unit ball bundle $B^* (\partial \Omega)^o$ to
be the projection to $T^* \partial \Omega$ of the inward pointing
unit vectors to $\Omega$ along $(\partial \Omega)^o$. We leave it
undefined at the singular points.

For concreteness, here we assume $n=2$ and write the smooth part of the boundary as a disjoint union
$ (\partial \Omega)^o = \bigcup_{j=1}^M W^o_{j},$
where the $W^o_{j}$ are open boundary faces diffeomorphic to open intervals of $\R.$ We let $ r_j: (a_j,a_j+1) \rightarrow W^o_j, \,\, s \mapsto r_j(s)$ denote unit-speed parametrizations of the boundary faces  with $a_0=-\pi, a_M= \pi$ and let
$$y: (a_j,a_{j+1})  \rightarrow H_j \subset H, \,\, s \mapsto Y_j(s),$$ be the parametrization defining the glancing set relative to $H_j.$ Here, the $H_j$'s are just open sub-arcs of $H.$    For fixed small $\eo >0$ let $\chi_j^{\eo} \in C^{\infty}_{0}(\R)$ be a cutoff equal to $1$ on $(a_j +\eo, a_{j+1}-\eo)$ for some boundary face indexed by $j \in \{1,...,M\}.$
It follows that
\begin{align} \label{corner2}
& h^{-1/2}\int \int_{S_{\eo,\pi}} e^{-2S(t)/h} | N^{\C}(h)  \chi_j^{\eo} \phi_h^{\partial \Omega}(t)|^{2} \, \chi_{\eo}(t) \, dt d\overline{t} \\ 
\nonumber&  = (2\pi h)^{-3/2}  \int \int_{S_{\eo,\pi}} \left( \int_{\R} \int_{\R} e^{i\Psi(t,s,s')/h}\chi_j^{\eo}(s) \chi_j^{\eo}(s') \,  \phi_h^{\partial \Omega}(s) \overline{\phi_h^{\partial \Omega}(s')} \, ds ds'  \right) \, dt d\overline{t} \end{align}

The analysis of the last integral on the RHS of  (\ref{corner2})  follows exactly  as in section \ref{opanalysis} and one gets that
$$ h^{-1/2} \int \int_{S_{\eo,\pi}} e^{-2S(t)/h} N^{\C}(h)  \chi_j^{\eo} \phi_h^{\partial \Omega}(t)  \cdot \overline{N^{\C}(h)  \chi_k^{\eo} \phi_h^{\partial \Omega}(t)}  \, \chi_{\eo}(t) \, dt d\overline{t} = \langle Op_h(a_{\gcal}^{(j)}) \phi_h^{\partial \Omega}, \phi_{h}^{\partial \Omega} \rangle_{L^2},$$
where,
$ \text{supp} a_{\gcal}^{(j)} \in C^{\infty}_{0}(B^* W_j^o)$ with $\int_{B^*W_j^o} a^{(j)}_{\gcal}(s,\sigma) \gamma(s,\sigma) ds d\sigma >0.$
Thus, Proposition \ref{mainthm2prop} follows also in the case where $\partial \Omega$ is only piecewise smooth. Theorem \ref{mainthm2} then follows from Theorem \ref{mainthm1} as outlined in the introduction.\qed

 \end{document}